\documentclass[12pt]{article}
\usepackage{graphicx}
\usepackage[centering, margin={0.8in, 0.5in}, includeheadfoot]{geometry}

\usepackage{graphicx}
\usepackage[utf8x]{inputenc}
\usepackage{amssymb,pdfsync,epstopdf,verbatim,gensymb,amsmath,amsthm}
\usepackage{textcomp}
\usepackage{amsmath}
\usepackage{subfigure}
\usepackage{amsfonts}
\usepackage{epsfig,color, hyperref}
\usepackage{verbatim,amsthm}  
\usepackage{url,verbatim}
\usepackage{graphicx,color}

\newcommand{\wh}{\widehat}

\newcommand{\lb}{\left(}

\newcommand{\rb}{\right)}

\newcommand{\wt}{\widetilde}

\newcommand{\Rb}{\mathbb{R}}

\renewcommand{\O}{\Omega}
\renewcommand{\L}{\langle}

\newcommand{\n}{\nabla}

\newcommand{\A}{\alpha}

\usepackage{amsthm}
\usepackage{amsfonts}

\def\H01{{H_0^1(\Omega)}}
\def\L2{{L^2(\Omega)}}

\usepackage{bmpsize}

\DeclareMathAlphabet{\bi}{OML}{cmm}{b}{it}
\DeclareMathAlphabet{\bcal}{OMS}{cmsy}{b}{n}
\DeclareMathAlphabet{\brmn}{OT1}{cmr}{bx}{n}
\usepackage{amsopn}
\usepackage{amsmath}

\newtheorem{theorem}{Theorem}[section]

\newtheorem{lemma}[theorem]{Lemma}
\newtheorem{proposition}[theorem]{Proposition}

\theoremstyle{definition}

\graphicspath{{./figs/}}
	
\begin{document}
			\title{A fully non-linear optimization approach to acousto-electric tomography}
	
	\author{B. J. Adesokan$^{*}$, K. Knudsen$^{*}$, V. P. Krishnan$^{\dagger}$ and S. Roy$^{\dagger}$}

\date{}	
\maketitle

\begin{abstract} This paper considers the non-linear inverse problem of reconstructing an electric conductivity distribution from the interior power density in a bounded domain. Applications include the novel tomographic method known as acousto-electric tomography, in which the measurement setup in Electrical Impedance Tomography is modulated by ultrasonic waves thus giving rise to a method potentially having both high contrast and high resolution. We 
formulate the inverse problem as a regularized non-linear optimization problem, 
show the existence of a minimizer, and derive optimality conditions. We propose a non-linear conjugate gradient scheme for finding a minimizer based on the optimality conditions. 
	All our numerical experiments are done in two-dimensions. The experiments reveal new insight into the non-linear effects in the reconstruction. One of the interesting features we observe is that, depending on the choice of regularization, there is a trade-off between high resolution and high contrast in the reconstructed images. 
	Our proposed non-linear optimization framework can be generalized to other hybrid imaging modalities. 

	\end{abstract}
\section{Introduction}

Hybrid tomography refers to a combination of two or more existing imaging modalities.  Several modalities such as X-ray Computed Tomography (CT), Ultrasound Imaging (UI), Magnetic Resonance Imaging (MRI) offer high resolution but have poor contrast in some situations. Other imaging modalities such as Electrical Impedance Tomography (EIT) and Optical Tomography (OT) have the reverse properties, that is, they offer high contrast in various applications, but suffer from poor resolution. By combining two modalities with different nature, one can hope to achieve a tomographic modality with  both high-contrast with  high-resolution. A partial list of modalities for hybrid tomography includes Impedance-acoustic Tomography (IAT) \cite{Gebauer-Scherzer} (coupling of EIT and UI),  Acousto-electric tomography (AET) \cite{ABCTF,Zhang-Wang} (coupling of EIT and UI), Photoacoustic tomography (PAT) \cite{Kuchment-Kunyansky-2} (coupling of OT and UI), and Magnetic resonance EIT (MREIT) \cite{SJAH,KKSY,KKSW} (coupling of MRI and EIT). For an overview of the several hybrid imaging modalities for conductivity imaging we refer the reader to \cite{Widlak-Scherzer}.

In this paper we focus on a computational approach to the hybrid imaging problem relevant to  AET. Mathematically the problem is as follows: Let $\Omega\subset \mathbb{R}^n$ be an open, bounded, {convex set  with smooth boundary.} The interior conductivity distribution is given by  a scalar function bounded above and below by positive constants. 
The application of a voltage potential $f$ to the boundary $\partial \Omega$ generates an interior voltage potential $u$ that is characterized by the elliptic PDE
\begin{equation}\label{elliptic_BVP}
\begin{aligned}
-\nabla\cdot(\sigma\nabla{u})&=0~ \mbox{in }\Omega, \\
u|_{\partial\Omega}&=f.\\
\end{aligned}
\end{equation}
In EIT one measures the normal current flux through the boundary given by $\sigma \nabla u \cdot \nu,$ with $\nu$ denoting the outward unit normal on $\partial\Omega.$ 
Ultrasound waves generated in the exterior of $\Omega$ can be used to perturb the interior conductivity due to the acousto-electric effect, and by measuring the resulting perturbed boundary current flux one can, in principle, compute the interior power density \cite{ABCTF,Bal2013}
\begin{equation*}
  H(\sigma)=\sigma|\n u|^{2} \text{ in } \Omega.
\end{equation*}

The inverse problem in AET is to uniquely determine and reconstruct the conductivity $\sigma$ from several power densities
\begin{equation}\label{Internal functional}
H_{i}(\sigma)=\sigma|\n u_{f_{i}}|^{2}, \mbox{ for } 1\leq i\leq m,
\end{equation}
where $u_{f_{i}}$ is the unique solution to \eqref{elliptic_BVP} with boundary potential $f_{i}$. As one can  easily see, the problem is a non-linear inverse problem.

For the two dimensional ($n=2$) problem uniqueness is known \cite{CFdGK} for  any three ($m=3$) boundary conditions   $f_1 ,f_2$ and $f_3 = f_1+f_2$ provided that the interior gradient fields satisfy
\begin{align}\label{eq:determinant}
  \det \left[ \nabla u_1, \nabla u_2 \right] \geq C> 0.
\end{align}
This conditions state that $u_1,u_2$ has no critical points and that $\nabla u_1$ and $\nabla u_2$ are nowhere collinear. This condition is satisfied for instance for $f_1=x_1,\; f_2 = x_2$ written in Cartesian coordinates $x = (x_1,x_2),$ but in fact any two boundary conditions $f_1,f_2$ that are almost two-to-one can be taken together with $f_3$ \cite{Alessandrini-Nessi}. In dimensions $n\geq 3$ the same question is a much more delicate issue \cite{alberti2017critical,capdeboscq2015}.

{The non-linear inverse problem has been analyzed mainly from a theoretical point of view, see \cite{Bal_APDE,CFdGK,BBMT,Monard-Bal-InverseDiffusion} for a partial list of works in this direction.
One approach for studying the non-linear problem is to consider the linearized problem. This has been analyzed both theoretically and numerically, see for example, \cite{Kuchment-Steinhauer1,Kuchment-Steinhauer2,Bal_Contemporary_Math,Montalto_Stefanov,Bal_Naetar_Scherzer_Schotland,Hoffmann_Knudsen,Kuchment-Kunyansky-AET,HubmerKnudsenLiSherina2018,LiaKaramehmedovicSherinaKnudsen2018}.  It can be shown that the linearized problem {in $\Rb^{2}$} is (microlocally) solvable in case of only two boundary conditions \cite{Bal_Contemporary_Math}. However, if the interior gradient fields  $\nabla u_1, \nabla u_2$ are somewhere orthogonal in the interior, then local instabilities occur and the inversion allows propagation of singularities \cite{BHK}. Consequently, the particular choice of boundary conditions  turns out to be crucial. }

{In our work, we consider a fully non-linear approach to the optimization problem. While there are several works, most notably \cite{ABCTF, CFdGK}, that have considered non-linear approaches to the reconstruction problem, to the best of our knowledge, ours is the first work that explicitly considers a regularized bilinear least squares optimization framework in the context of AET. The main novelty of the paper is that we provide a non-linear computational framework that has the potential for reconstructing  conductivities with better contrast as well as resolution. In this context, a computational approach using edge-enhancing techniques for AET been done recently in \cite{Roy-Borzi,AdesokanJinJensenKnudsen2018}.}

We will, as in \cite{Bal_Naetar_Scherzer_Schotland}, assume  that $\sigma \in H^s(\Omega)$ with $s> \frac{n}{2}$ an integer. Then $H^s(\Omega)$ is a Banach algebra and $\sigma \in C(\overline \Omega).$ This is a rather strong regularity assumption that allows our theoretical analysis below, but most likely the results can be extended to less regular conductivities. We take two boundary conditions $m=2$ such that \eqref{eq:determinant} is satisfied. For the two dimensional problem we conjecture that for such two well-chosen boundary conditions the non-linear problem is uniquely solvable, however, we will not attempt to  prove this. Instead we take a computational approach to the fully non-linear problem. We cast the inverse problem as a {bilinear} optimization problem, show existence of a minimizer and  develop a non-linear conjugate gradient (NLCG) optimization approach for the reconstruction.

The outline of the paper is as follows: In Section \ref{sect:NLCG} we formulate the optimization problem, show existence of solutions and derive optimality conditions. In Section \ref{sect:discrete} we discretize the optimality system and outline the NLCG approach. In Section \ref{sect:numerics}, we describe the numerical implementation and carry out several computational experiments. We conclude in Section \ref{sect:conclusion}.

\section{The optimization problem and its minimizer}\label{sect:NLCG}

We consider an optimization-based approach for reconstructing $\sigma$ given $H_1(\sigma),H_2(\sigma)$. For  $\sigma \in H^s(\Omega)$, the power density function $H(\sigma)$ also belongs to $H^s(\Omega)$ \cite{Bal_Naetar_Scherzer_Schotland}, and thus it makes sense to consider the following cost functional: 
\begin{equation}\label{cost_functional}
J(\sigma, u_1, u_2) = \dfrac{1}{2}\|\sigma|\nabla u_1|^{2}-H_1^{\delta}\|^2_{L^{2}(\Omega)} + \dfrac{1}{2}\|\sigma|\nabla u_2|^{2}-H_2^{\delta}\|^2_{L^{2}(\Omega)}+ \dfrac{\alpha}{2} \|\sigma-\sigma_b\|^2_{H^{s}(\Omega)}.
\end{equation}
In the above equation, $u_1$ and $u_2$ satisfy \eqref{elliptic_BVP} with boundary data $f_1$ and $f_2$, respectively, and $\sigma_b\in H^{s}(\O)$ is a chosen background conductivity. The quantities $H_{1}^{\delta}, H_{2}^{\delta}\in L^{2}(\O)$ denote the power density functionals possibly  corrupted with noise. 
We will reconstruct $\sigma$ in the following admissible set 
\[
H^{s}_{\mathrm{ad}}(\O):=\{\sigma\in H^s(\O) \mbox{ such that } {0<\sigma_l\leq\sigma(x)\leq\sigma_u} \mbox{ for all } x\in \O\},
\] (here $\sigma_l$ and $\sigma_u$ are given positive constants)  by considering the minimization problem: 
\begin{equation}\label{min_problem}
\begin{aligned}
\min_\sigma & ~J(\sigma, u_1, u_2),\\
\mbox{ such that } &\mathcal{L}_{f_i}(u_i,\sigma)=0,\; i=1,2.
\end{aligned} \tag{P}
\end{equation}
In the rest of the paper, we consider $s=\lfloor \frac{n}{2} \rfloor +1$, where $\lfloor\cdot \rfloor$ denotes the greatest integer function. 
The equality $\mathcal L_f(u,\sigma)=0$ is a short-hand notation for \eqref{elliptic_BVP}. 
In this section, we discuss the existence of solutions to the minimization problem (\ref{min_problem}) and state the optimality system for the characterization of a minimizer. 


\subsection{Existence of a minimizer}

Our analysis of the minimization problem (\ref{min_problem}) begins with the discussion of the existence of solution of (\ref{elliptic_BVP}) which is proved in \cite{solonnikov1973a}.
\begin{proposition}\label{existence_regularity}
Let $\sigma\in H^{s}_{\mathrm{ad}}$ and $f \in H^{s+1/2}(\partial \Omega)$. Then \eqref{elliptic_BVP} has a unique solution $u\in H^{s+1}(\Omega).$ 
\end{proposition}
We will denote this unique solution by $u(\sigma)$. Next we consider the  Fr\'{e}chet differentiability of the  mapping $u(\sigma)$ which is proved in \cite{Bal_Naetar_Scherzer_Schotland}. 

\begin{lemma}\label{differentiable_constraint}
The map $u(\sigma)$ defined by \eqref{elliptic_BVP} is Fr\'{e}chet differentiable as a mapping from $H^s(\Omega)$ to $H^{s+1}(\Omega)$.
\end{lemma}
{Using Lemma \ref{differentiable_constraint}, we introduce the reduced cost functional
\begin{equation}\label{reduced_functional}
\wh{J}(\sigma) = J(\sigma, u_1(\sigma), u_2(\sigma)),
\end{equation}
where $u_i(\sigma)$, $i=1,2$ denotes the unique solution of \eqref{elliptic_BVP} given $\sigma$ and $f_i,i=1,2$.

We next state some properties of the reduced functional $\wh{J}$ which can be proved using the arguments in \cite{Bal_Naetar_Scherzer_Schotland}.
\begin{proposition}
The reduced functional $\wh{J}$, given in (\ref{reduced_functional}), is weakly lower semi-continuous (w.l.s.c.), non-negative and Fr\'{e}chet differentiable as a function of $\sigma$. 
\end{proposition}
We are now ready to show the existence of a minimizer of the optimization problem (\ref{min_problem}) using the reduced functional $\wh{J}$. In the statement of the theorem below, we denote $H^{s}_{f}(\O)$ as the closed convex subset of $H^{s}(\O)$ with boundary trace $f$.

\begin{theorem}
Let $f_1,f_2 \in H^{s+1/2}(\partial\Omega)$. Then there exists a triplet $(\sigma^*,u_1^*,u_2^*) \in H^{s}_{\mathrm{ad}}(\O)\times H^{s+1}_{f_1}(\Omega)\times H^{s+1}_{f_2}(\Omega)$ such that $u_i^*,~ i=1,2$ are solutions to {$\mathcal{L}_{f_i}(u_i,\sigma)=0,~ i=1,2$} and $\sigma^*$ minimizes $\wh{J}$ in $H^{s}_{\mathrm{ad}}(\O)$.
\end{theorem}
\begin{proof} We have that the 
boundedness from below of $\wh{J}$ guarantees the existence of a minimizing sequence
$\{\sigma^m\}\in H^{s}_{\mathrm{ad}}(\O)$ and since $\wh{J}$ is coercive, this sequence is bounded. Therefore it contains a weakly convergent
subsequence $\{\sigma^{m_l}\}$ in $H^{s}_{\mathrm{ad}}(\O)$ such that $\sigma^{m_l} \rightharpoonup \sigma^*$ (say). Since $H^{s}_{\mathrm{ad}}(\O)$ is weakly closed, we have that $\sigma^*\in H^{s}_{\mathrm{ad}}(\O).$ Since $\{\sigma^{m_l}\}$ is a minimizing sequence for $\wh{J}$, we obtain the sequence $(u_1^{m_l},u_2^{m_l})$, where $u_i^{m_l}=u_i(\sigma^{m_l})$, which is bounded in $ H_{f_{1}}^{s+1}(\Omega)\times H_{f_{2}}^{s+1}(\Omega)$. 
This implies that the sequence converges weakly to (say) $(u_1^*,u_2^*) \in H_{f_{1}}^{s+1}(\Omega)\times H_{f_{2}}^{s+1}(\Omega)$. 

{
We next show that the sequence $(\sigma^*,u_1^*,u_2^*)$ is a weak solution of (\ref{elliptic_BVP}). First note that the triplet $(\sigma^{m_l},u_1^{m_l},u_2^{m_l})$ is a weak solution of (\ref{elliptic_BVP}) for all $m_l\in \mathbb{N}$, that is $\langle \sigma^{m_l}\nabla u_i^{m_l},\nabla v\rangle_{L^2(\Omega)} = 0$ for any $v\in H^{1}_0(\Omega)$.
Now, since $H^{s}(\O)$ is compactly embedded in $L^4(\O)$, we have that $\sigma^{m_l}$ and $\nabla u_i^{m_l}$ converges strongly to $\sigma^*$ and $u_i^*$ respectively in $L^4(\O)$. Consequently $\sigma^{m_l}\nabla u_i^{m_l}$ converges strongly to $\sigma^*\nabla u_i^*$ in $L^2(\Omega).$ Hence $0 = \langle \sigma^{m_l}\nabla u_i^{m_l},\nabla v\rangle_{L^2(\Omega)} \rightarrow \langle \sigma^*\nabla u_i^*,\nabla v\rangle_{L^2(\Omega)}$ all for $v\in H_0^1(\Omega)$ showing that $(\sigma^*,u_1^*,u_2^*)$ is the unique solution of (\ref{elliptic_BVP}).}

Now by w.l.s.c. of $J$, we have 
$$
\wh{J}(\sigma^*) \leq \liminf_{m_l\rightarrow \infty} \wh{J}(\sigma^{m_l}) = \inf_{\sigma\in H^{s}_{\mathrm{ad}}(\O)} \wh{J}(\sigma).
$$
Thus, $\sigma^*$ minimizes the reduced functional $\wh{J}$ and this proves the existence of a minimizer of the optimization problem (\ref{min_problem}).
\end{proof}
}

\subsection{The reduced functional and optimality conditions}
In this section, we state the first order necessary optimality conditions for the minimizer of \eqref{cost_functional}.

Correspondingly, a local minimum $\sigma^*\in H^s_{\mathrm{ad}}(\Omega)$ of $\wh{J}$ is characterized by the first-order necessary optimality conditions given by 
\[
\Big{\langle}\nabla \wh{J}(\sigma^*),\widetilde{\sigma}-\sigma^*\Big{\rangle}_{L^{2}(\O)} \geq 0, \mbox{for all }\widetilde{\sigma} \in H^{s}_{\mathrm{ad}}(\O),
\]
where $\nabla \widehat{J}(\sigma^*)$ denotes the $L^2(\Omega)$ gradient and in the inner product above, we interpret $\n \wh{J}(\sigma^{*})$ as the Riesz representative of the Frech\'et derivative of $\wh{J}$ in $L^{2}$ evaluated at $\sigma^{*}$.  
It is well known (see for instance \cite{neit:tiba}) that using the Lagrange functional,
\[
L(\sigma, u_1, u_2, v_1, v_2) = J(\sigma,u_1,u_2) + \langle \sigma \nabla u_1,\nabla v_1\rangle_{L^{2}(\O)} + \langle \sigma \nabla u_2,\nabla v_2\rangle_{L^{2}(\O)},
\] in the framework of the adjoint method, the condition 
$\Big{\langle}\nabla \widehat{J}(\sigma^*),\widetilde{\sigma}-\sigma^*\Big{\rangle}_{L^{2}(\O)} \geq 0$, results in the following optimality system, consisting of the forward and adjoint equations and a variational inequality. We have
\begin{align}
\label{for_1}& -\nabla\cdot(\sigma\nabla{u_1})=0~ \mbox{in }\Omega, \quad
u_1|_{\partial\Omega}=f_1,\\
\label{adj_1}&-\nabla\cdot(\sigma\nabla{v_1})=2\nabla\cdot(\sigma [\sigma |\nabla u_1|^2-H_1^{\delta}]\nabla u_1)~ \mbox{in }\Omega, \quad
v_1|_{\partial\Omega}=0,\\
\label{for_2}&-\nabla\cdot(\sigma\nabla{u_2})=0~ \mbox{in }\Omega, \quad
u_2|_{\partial\Omega}=f_2,\\
\label{adj_2}&-\nabla\cdot(\sigma\nabla{v_2})=2\nabla\cdot(\sigma [\sigma |\nabla u_2|^2-H_2^{\delta}]\nabla u_2)~ \mbox{in }\Omega, \quad
v_2|_{\partial\Omega}=0,\\
\notag &\Big{\langle} (\sigma |\nabla u_1|^2-H_1^{\delta})|\nabla u_1|^2+(\sigma |\nabla u_2|^2-H_2^{\delta})|\nabla u_2|^2 + \\
&\label{optimality}\quad \quad\alpha\sum_{k=0}^{s} (-1)^k \Delta^k(\sigma-\sigma_b) + \nabla u_1 \cdot \nabla v_1 + \nabla u_2 \cdot \nabla v_2, \widetilde{\sigma}-\sigma \Big{\rangle}_{L^{2}(\O)} \geq 0,
\end{align}
for all $\wt{\sigma}\in H^{s}_{\mathrm{ad}}(\O)$.

\section{Discretization of the optimality system}\label{sect:discrete}
\subsection{Numerical discretization of the forward and adjoint problems}

In this section, we discuss the numerical approximation to the forward and adjoint elliptic equations in (\ref{for_1})--(\ref{adj_2}) using the finite element method. We describe the discretization schemes for solving (\ref{for_1})-(\ref{adj_1}). The same schemes would be used for (\ref{for_2})-(\ref{adj_2}).
 We first note that for a two-dimensional or three-dimensional setup $s=\lfloor \frac{n}{2} \rfloor +1=2$, i.e. $\sigma\in H^2(\Omega)$. This implies that the regularization term in (\ref{cost_functional}) is $\dfrac{\alpha}{2} \|\sigma-\sigma_b\|^2_{H^{2}(\Omega)}$. Consequently, the left hand side of (\ref{optimality}) involves a  fourth order PDE and is computationally very expensive to solve. Therefore, in the numerical simulations below, we use lower order regularization terms to determine the optimality system and use the NLCG method with the corresponding reduced gradient. More specifically, we use $L^2$ and $H^1$ regularization terms (corresponding to $s=0$ and $s=1$ respectively). The corresponding reduced gradients used in the NLCG method are the $L^2$ and the $H^1$ gradients. We emphasize that though the optimal solution $\sigma^*$ obtained through this procedure is less regular, the method is computationally efficient. Furthermore, using the $H^1$ gradient, we have a good approximation of the desired $\sigma\in H^2(\Omega)$.

The weak form representation of (\ref{for_1}) is as follows: Find $u\in H^{1}(\Omega)$ with boundary trace $f$ such that
\begin{equation}\label{weak_form}
\int_\Omega~\sigma\nabla u\cdot\nabla \widetilde{u} =0
\end{equation}
for all $\widetilde{u} \in H^{1}_0(\Omega)$.
Let us define the space of continuous functions which are piecewise polynomials of degree $k$ in a triangle element $K$ belonging to a mesh $\tau_h$ as follows
\begin{equation}\label{space}
W_{f,h}^{k}(\O)=\lbrace u_h\in C^{0}(\overline{\Omega}):u_h|_K\in\mathbb{P}_k\mbox{ for all }  K\in\tau_h\rbrace\cap \lbrace u_h=f \mbox{ on } \partial \Omega \rbrace.
\end{equation}
We also define the bilinear form
\begin{equation}\label{bilinear_form}
A(u,\widetilde{u}) = \int_\Omega~\sigma\nabla u\cdot\nabla \widetilde{u}.
\end{equation}
Then the discrete scheme for (\ref{for_1}) is given as follows: Find $u_h \in W_{f,h}^{k}(\O)$, such that
\begin{equation}\label{vel_scheme}
A(u_h,\widetilde{u}_h)= 0,
\end{equation}
for all $\widetilde{u}_h \in {W}_{0,h}^{k}(\O)$, where

	
	\begin{equation}\label{space_adj}
	{W}_{0,h}^{k}=\lbrace u_h\in C^{0}(\overline{\Omega}):u_h|_K\in\mathbb{P}_k\quad\forall K\in\tau_h\rbrace\cap \lbrace u_h=0 \mbox{ on } \partial \Omega \rbrace.
	\end{equation}
For the adjoint equation (\ref{adj_1}), we define the linear form
\begin{equation}\label{linear_form}
L(\widetilde{v}) = -2\int_\Omega~(\sigma [\sigma |\nabla u_{h}|^2-H_1^{\delta}]\nabla u_{h})\cdot\nabla \widetilde{v},
\end{equation}
where $\nabla u_{h}$ is the derivative of the solution $u_{h}$ to (\ref{vel_scheme}).
Then the discrete scheme for (\ref{adj_1}) is given as follows: Find $v_h \in W_{0,h}^k$, such that
\begin{equation}\label{adj_scheme}
A(v_h,\widetilde{v}_h)= L(\widetilde{v}_h),
\end{equation}
for all $\widetilde{v}_h \in {W}_{0,h}^{k}$ defined in \eqref{space_adj}, and $A(u,v)$ is the bilinear form defined in (\ref{bilinear_form}).

\subsection{The reduced $H^1$ gradient}

For the case $s=1$, in the optimality system (\ref{optimality}), the following reduced $L^2$ gradient components appear
\begin{equation}\label{l2grad}
\begin{aligned} 
&\nabla \widehat{J}(\sigma) =\Big[ \lb\sigma |\nabla u_1|^2-H_1^{\delta}\rb|\nabla u_1|^2+\lb\sigma |\nabla u_2|^2-H_2^{\delta}\rb|\nabla u_2|^2 + \\
&\quad\quad\quad\quad\quad \alpha(\sigma-\sigma_b)-\alpha \Delta(\sigma-\sigma_b) + \nabla u_1 \cdot \nabla v_1 + \nabla u_2 \cdot \nabla v_2\Big],
\end{aligned}
\end{equation}
where $\Delta$ is the distributional Laplacian.
Let  us now discuss the unconstrained case. In this 
case, optimality requires $\nabla \widehat{J}(\sigma) =0$. Because of the $H^1$ cost for $\sigma-\sigma_b$, we have a setting that allows to include boundary conditions
on the conductivity $\sigma$. By considering the derivation of 
the optimality system above using the Lagrange formulation, we  find that a convenient choice is to require $\sigma-\sigma_b=0$ on $\partial \Omega$ as the conductivity distribution near the boundary is constant and equals to the background distribution $\sigma_b$.

We wish to apply a gradient-based optimization scheme where the residual of \eqref{l2grad} is used such that $\sigma\in H^1(\Omega)$. For this purpose, we cannot use this residual directly for updating the conductivity, since it 
is not in $H^{1}(\Omega)$. Therefore, it is necessary 
to determine the reduced $H^1$ gradient. This is done based on the following fact 
$$
\Big\langle\nabla \widehat J(\sigma)_{H^1(\O)}, \varphi \Big\rangle_{H^1(\Omega)}=
\Big\langle\nabla \widehat J(\sigma), \varphi\Big\rangle_{L^2(\Omega)},
$$
where $\varphi\in H_0^1(\Omega)$.
Using the definition of the $H^1$ inner product and integrating by
parts, we have that the $H^1$ gradient is obtained by solving the 
following boundary value problem 
\begin{eqnarray}
-\Delta (\nabla\widehat J(\sigma)_{H^1(\O)}) +\nabla\widehat J(\sigma)_{H^1(\O)} = \nabla\widehat J(\sigma)  \mbox{ in }  \Omega \label{gradH1x}  \\
\nabla\widehat J(\sigma)_{H^1(\O)} =0   \mbox{ on }  \partial \Omega.  \label{gradH1xBC} 
\end{eqnarray}
where (\ref{gradH1x})-(\ref{gradH1xBC}) is defined in the weak sense.
The solution to this problem provides the appropriate gradient to be used 
in a gradient update of the conductivity that includes projection to ensure $\sigma\in H^{1}_{\mathrm{ad}}(\O)$. 

\subsection{A projected NLCG optimization scheme\label{sec:optimization}}

 We solve the optimization problem (\ref{min_problem}) by implementation of a projected non-linear conjugate scheme (NLCG); see \cite{neit:tiba} in $L^2$ and $H^1$ spaces. Such a scheme is an extension of the conjugate gradient method to constrained non-linear optimization problems. In the following discussing we denote $X_h$ as both the discrete approximations to the $L^2(\O)$ and $H^1(\O)$ spaces. We also denote the corresponding discrete inner product and norm as $\langle\cdot,\cdot\rangle_{X_h}$ and $\|\cdot\|_{X_h}$, respectively, where$\|\cdot\|^2_{X_h}=\langle \cdot,\cdot\rangle_{X_h}$ . For the definition of the discrete $L_h^2,H_h^1$ inner product we refer to \cite{Suli}. To describe this iterative method, we start with an initial guess $\sigma_0$ for the conductivity and the corresponding search direction: 
\[
d_0=-g_0:= -(\nabla\widehat{J}(\sigma_0))_{X_h},
\]
where $\nabla\widehat{J}(\sigma_0)_{X_h}$ represents the discrete $L^2$ or $H^1$ gradient computed through a finite element discretization of (\ref{optimality}) or (\ref{gradH1x})--(\ref{gradH1xBC}), respectively. The search directions are obtained recursively as
\begin{equation}
d_{k+1} = -g_{k+1}+\beta_kd_k,
\end{equation}
where $g_k=\nabla\widehat{J}(\sigma_k)_{X_h},~k=0,1,\hdots$ and the parameter $\beta_k$ is chosen according to the formula of Hager-Zhang \cite{hag:zha}  given by
\begin{equation}\label{HG}
\beta_k^{HG} = \frac{1}{d_k^Ty_k}\left({y_k-2d_k\frac{\|y_k\|_{X_h}^2}{d_k^Ty_k}}\right)^Tg_{k+1},
\end{equation}
where $y_k = g_{k+1}-g_k$.

We update the value of the conductivity $\sigma$ with a steepest descent scheme given as follows
\begin{equation}\label{globGradOpt}
\sigma_{k+1} = \sigma_k + \alpha_k \,  d_k ,
\end{equation}
where $k$ is a index of the iteration step and $\alpha_k >0$ is a step length obtained using a line search algorithm as in \cite{MA}. For this line search, we use the following Armijo condition of sufficient decrease of $\widehat{J}$
\begin{equation}\label{arm}
\widehat{J}(\sigma_k+\alpha_kd_k)\leq \widehat{J}(u_k)+\delta\alpha_k\langle\nabla\widehat{J}(\sigma_k),d_k\rangle_{X_h},
\end{equation}
where $0 < \delta < 1/2$.

Notice that this gradient procedure should be combined with
a projection step onto $H^s_{\mathrm{ad}}$. Therefore, we consider the following
\begin{equation}\label{localGradOpt}
\sigma_{k+1}
= P_{L}\left[ \sigma_k + \alpha_k \, d_k \right] ,
\end{equation}
where
$$
P_{L}\left[\sigma\right] = \max\lbrace \sigma_l,\min \lbrace \sigma_u, \sigma\rbrace\rbrace.
$$

The projected NLCG scheme can be summarized as follows: 

\begin{enumerate}
\item Input: initial approximation, $\sigma_0$. Evaluate $d_0 = -\nabla\widehat{J}(\sigma_0)_{X_h}$, index $k=0$, maximum $k=k_{\mathrm{max}}$,
tolerance =tol.
\item While $(k<k_{max})$ do
\begin{enumerate}
\item Set $\sigma_{k+1} = P_{L}\left[ \sigma_k + \alpha_k \, d_k \right]$, where $\alpha_k$ is obtained using a line-search algorithm.
\item Compute $g_{k+1} = \nabla\widehat{J}(\sigma_{k+1})_{X_h}$.
\item Compute $\beta_k^{HG}$ using (\ref{HG}).
\item Set $d_{k+1}=-g_{k+1}+\beta_k^{HG}d_k$.
\item If $\|\sigma_{k+1}-\sigma_k\|_{X_h} < \mbox{tol.}$, terminate.
\item Set $k=k+1$.
\end{enumerate}
\item End while.
\end{enumerate}
\section{Numerical experiments}\label{sect:numerics}
In this section we discuss the numerical implementation of the NLCG scheme for the minimization problem (\ref{min_problem}). 
The domain of definition is the unit circle centered at $(0,0)$. We choose the value of the background conductivity as $\sigma_b=1.0$ and the lower and upper values $\sigma_l=0.01,~\sigma_u=4.0$. {The initial guess for $\sigma$ in the NLCG algorithm is chosen to be 1.} The computations are done in FENICS with $\mathbb{P}_2$ elements for the electric potential $u$ and $\mathbb{P}_1$ for the conductivity function $\sigma$ in case $s=0$. In the case $s=1$, we note that the optimality condition (\ref{optimality}) contains a Laplacian of $\sigma$ and thus we use $\mathbb{P}_2$ elements for $\sigma$. The average mesh size for the optimization algorithm is 0.01. The plot of the mesh is shown in Figure \ref{fig:mesh}. We choose the regularization parameter $\A=0.1$ for all the numerical experiments. 
\begin{figure}[ht]
\centering
\subfigure[Mesh]{\includegraphics[scale=0.32,keepaspectratio]{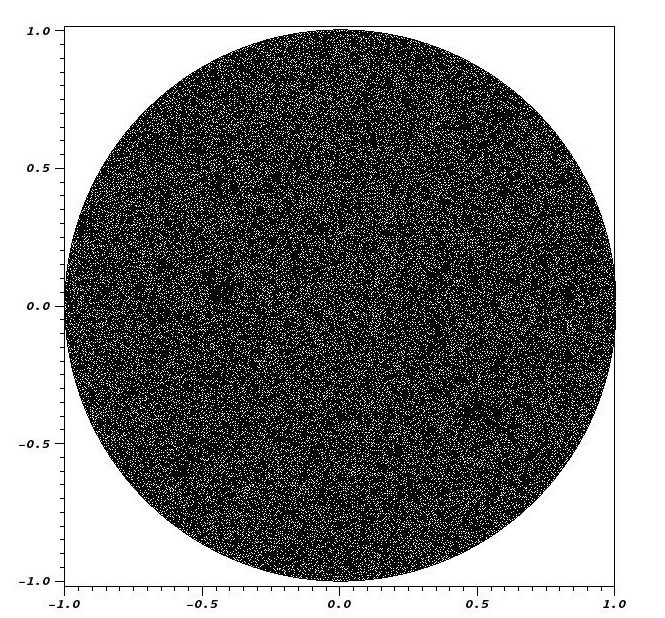}}
\subfigure[Zoomed view of the mesh]{\includegraphics[scale=0.32,keepaspectratio]{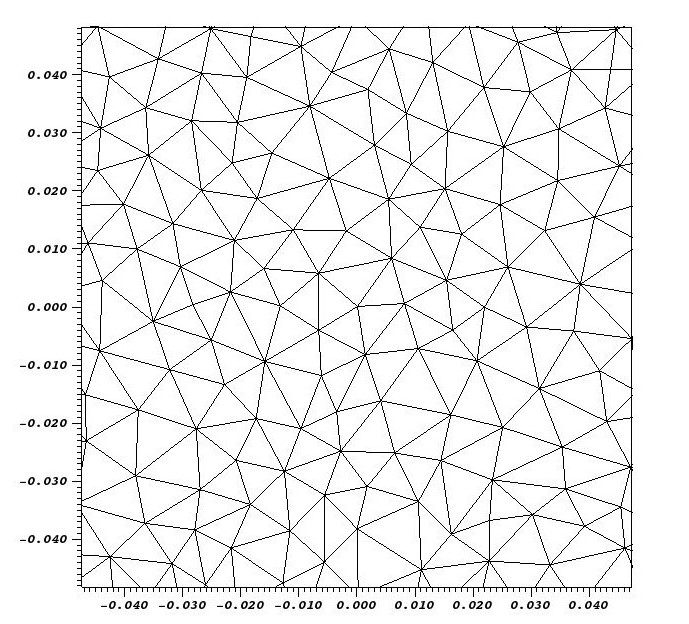}}
  \caption{The mesh for the experiments }
\label{fig:mesh}
  \end{figure}
 
 In our numerical simulations, we consider the following sets of boundary conditions:
 \begin{equation}\label{types_bc}
 \begin{aligned}
 &\mbox{BC1: } f_1=x,~f_2=\dfrac{x+y}{\sqrt{2}},\\
 &\mbox{BC2: } f_1=x,~f_2=y,\\
 &{\mbox{BC3: } f_1=x,~f_2=y,~f_3=\dfrac{x+y}{\sqrt{2}}.}\\
 \end{aligned}\tag{BC}
 \end{equation}
 With either of these choice of boundary conditions, $u_{1}$ and $u_{2}$ have no critical points and $\nabla u_1,\nabla u_2$ are non-parallel in $\overline{\Omega}$ \cite{Alessandrini-Nessi}. This choice of boundary conditions is motivated by the linear reconstruction algorithms, where BC1 and BC2 lead to different qualitative behavior in the reconstructions \cite{BHK}.

Unless otherwise explicitly stated, the boundary condition in the numerical experiments is BC1; see (\ref{types_bc}).
To generate the data $H$ in (\ref{Internal functional}),
we choose a $\sigma$ and solve (\ref{for_1}) on a finer mesh with mesh size $h = 0.005$. We then compute the gradients of $u$ using a finite element discretization and thus compute the internal data $H$. Finally, we project the data onto the computational mesh for our NLCG algorithm. 

{In general, we expect a better resolution of the reconstructions with $L^2$ regularization than with $H^1$ regularization.  Recall that at each iterative step, the update for $\sigma$ is found by solving \eqref{optimality}. In the case of $H^{1}$ regularization, since \eqref{optimality} involves an additional Laplacian term, the obtained update for $\sigma$ is more regular compared to that with the $L^{2}$ regularization set up. Due to this, the artifacts with $H^{1}$ regularization are less pronounced leading to reconstructions with better contrast. For the same reason, the edges are enhanced using $L^{2}$ regularization resulting in images with better resolution. Also note that more artifacts are present in images with $L^{2}$ regularization compared to that with $H^{1}$ regularization.
}

Test Case I: In the first test case, we consider a phantom represented by a disk and the conductivity $\sigma$ is defined as follows:

Let $r=\sqrt{(x-0.2)^2+(y-0.2)^2}$. Define
\begin{equation}
\sigma(x) =
\begin{cases}
2.0,&\qquad r < 0.3,\\
1.0,&\qquad r \geq 0.3.\\
\end{cases}
\end{equation}

The plots of the actual and reconstructed $\sigma$ with the boundary condition BC1 given in (\ref{types_bc}) {and with various values of $L^{2}$ and $H^{1}$ regularization parameter $\A$}  are shown in Figure \ref{disk_sigma}. {We observe that as $\A$ increases, the contrast in both the cases decreases. Regardless of the value of the regularization parameter $\A$, we observe better resolution with $L^{2}$ regularization and better contrast with $H^{1}$ regularization. We also compare our algorithm with the paramterix method of \cite{Kuchment-Kunyansky-AET} (shown in  Figure \ref{disk_parametrix}), and while there is a slightly better resolution of the edges compared to our algorithm, there is a substantial loss of contrast in the parametrix method.}

\begin{figure}[h]
	\centering
	\subfigure[Actual phantom]{\includegraphics[scale=0.17,keepaspectratio]{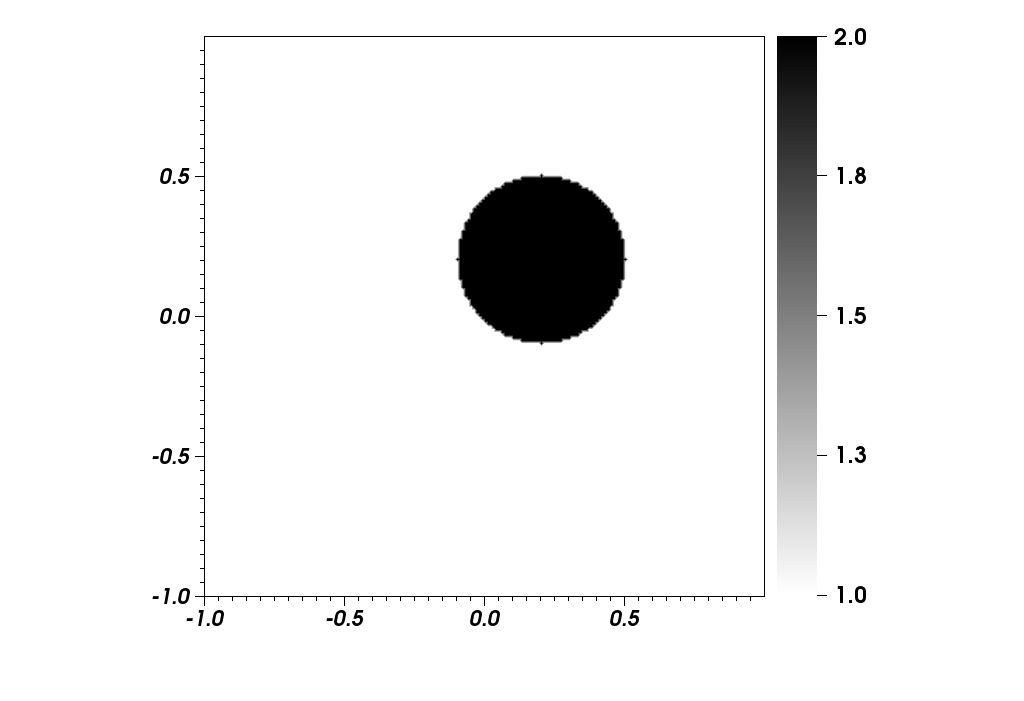}\label{disk_actual}}\hspace{5mm}
	\subfigure[ Parametrix method \cite{Kuchment-Kunyansky-AET} (Done in Matlab)]{\includegraphics[height=1.65in,width=1.95in]{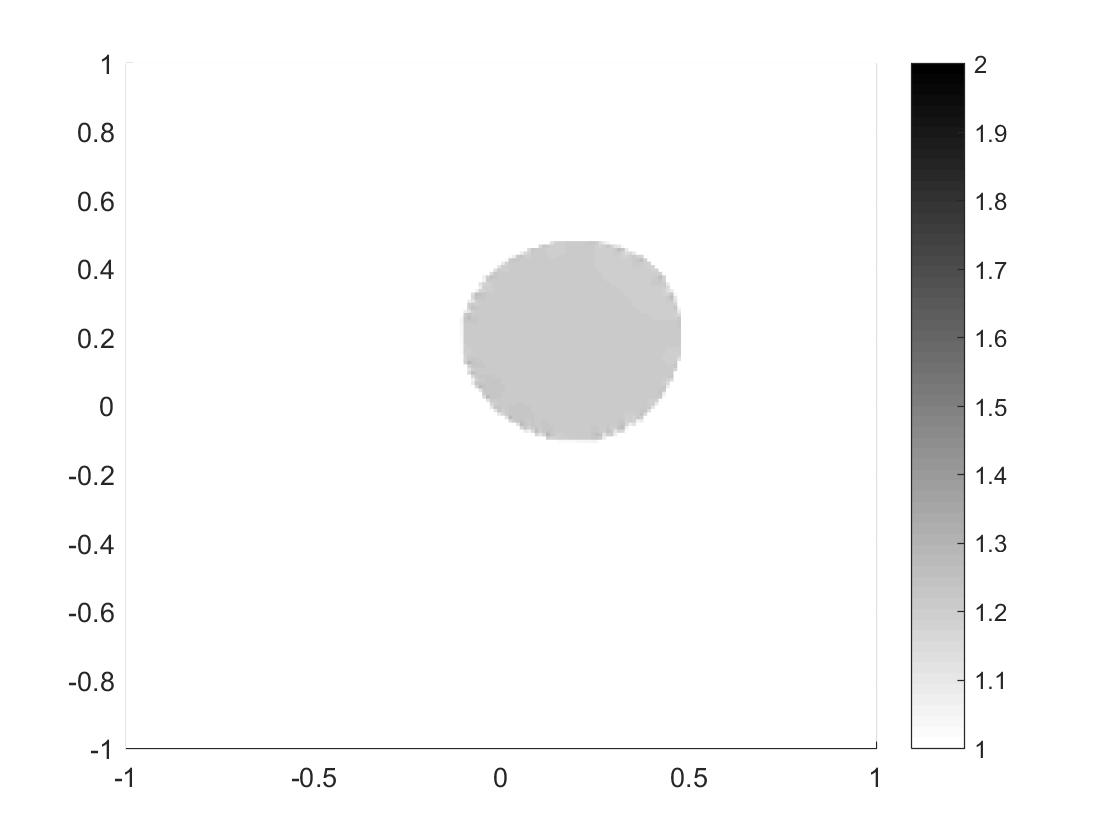}\label{disk_parametrix}}\\
	\subfigure[ $s=0,~\alpha=0.1$]{\includegraphics[scale=0.17,keepaspectratio]{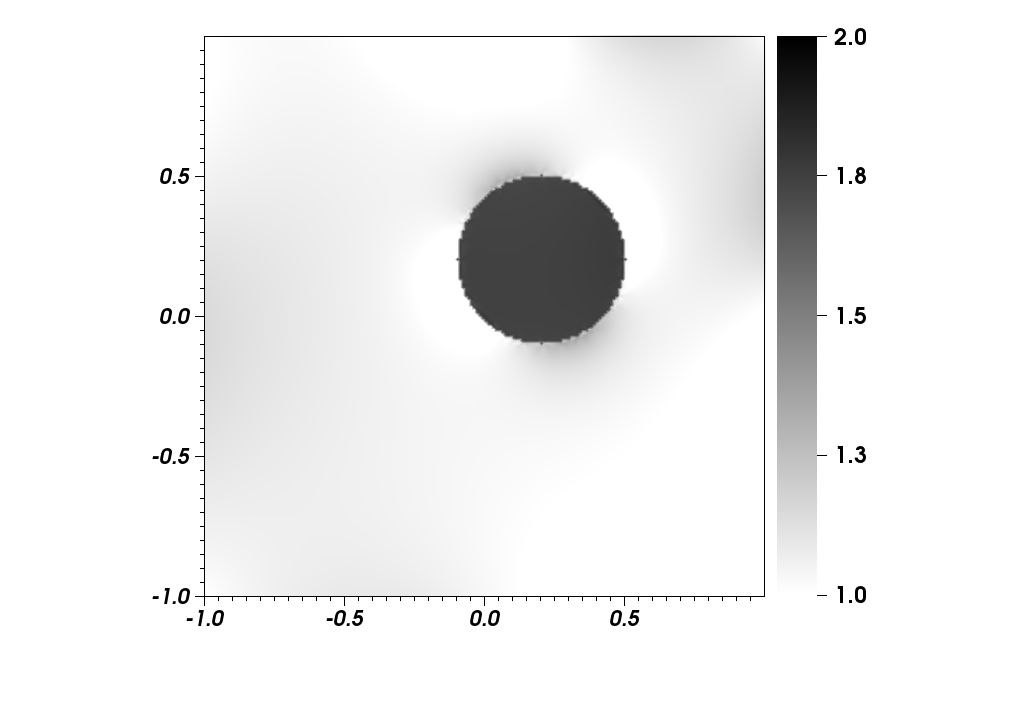}\label{disk_L2_1}}
	\subfigure[ $s=1,~\alpha=0.1$]{\includegraphics[scale=0.17,keepaspectratio]{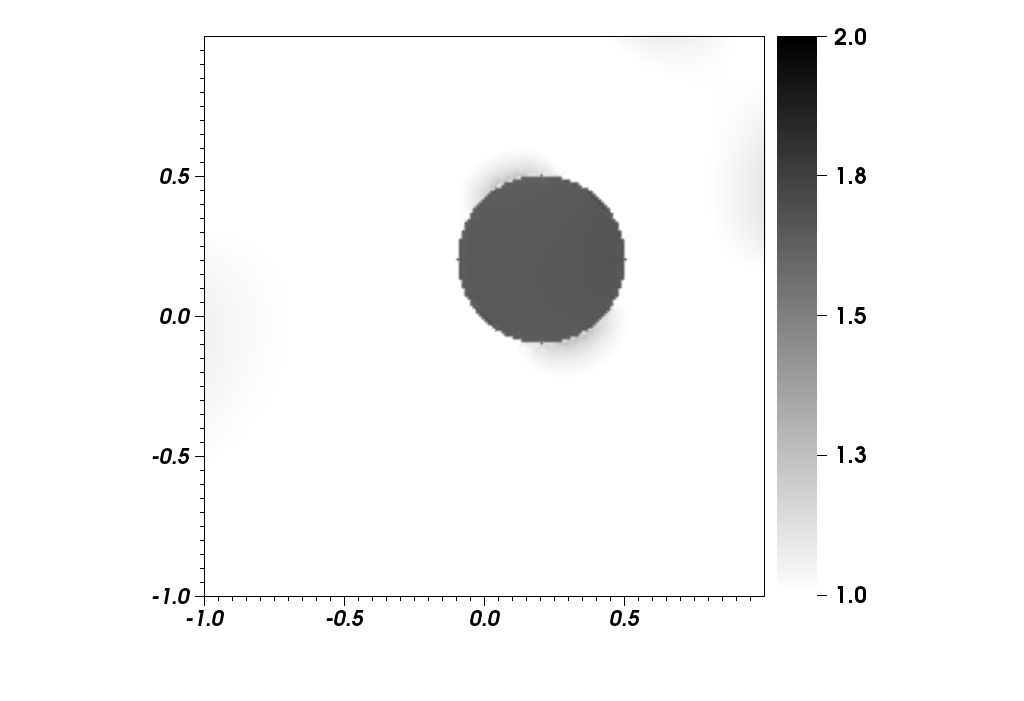}\label{disk_H1_1}}\\
    \subfigure[ $s=0,~\alpha=0.4$]{\includegraphics[scale=0.17,keepaspectratio]{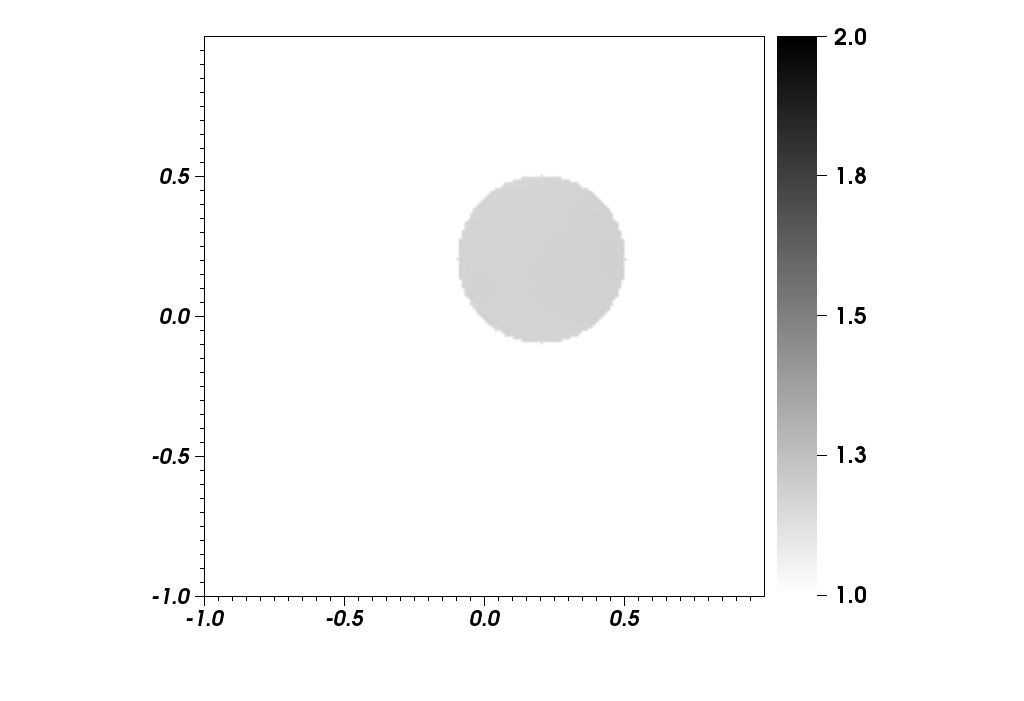}\label{disk_L2_4}}
    \subfigure[ $s=1,~\alpha=0.4$]{\includegraphics[scale=0.17,keepaspectratio]{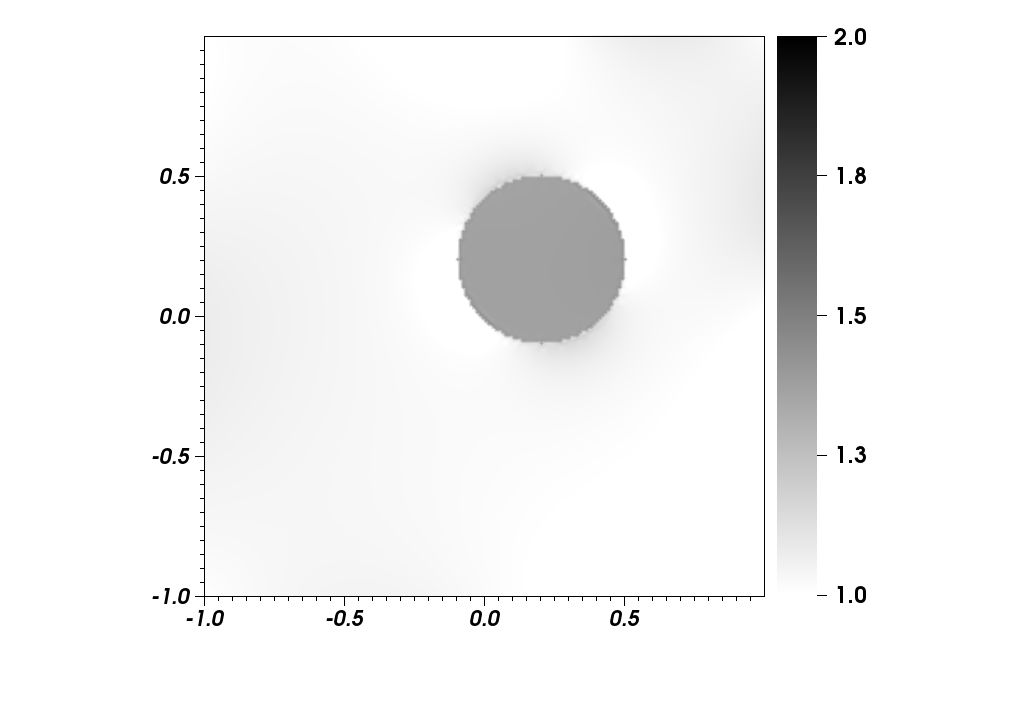}\label{disk_H1_4}}\\
    \subfigure[ $s=0,~\alpha=0.7$]{\includegraphics[scale=0.17,keepaspectratio]{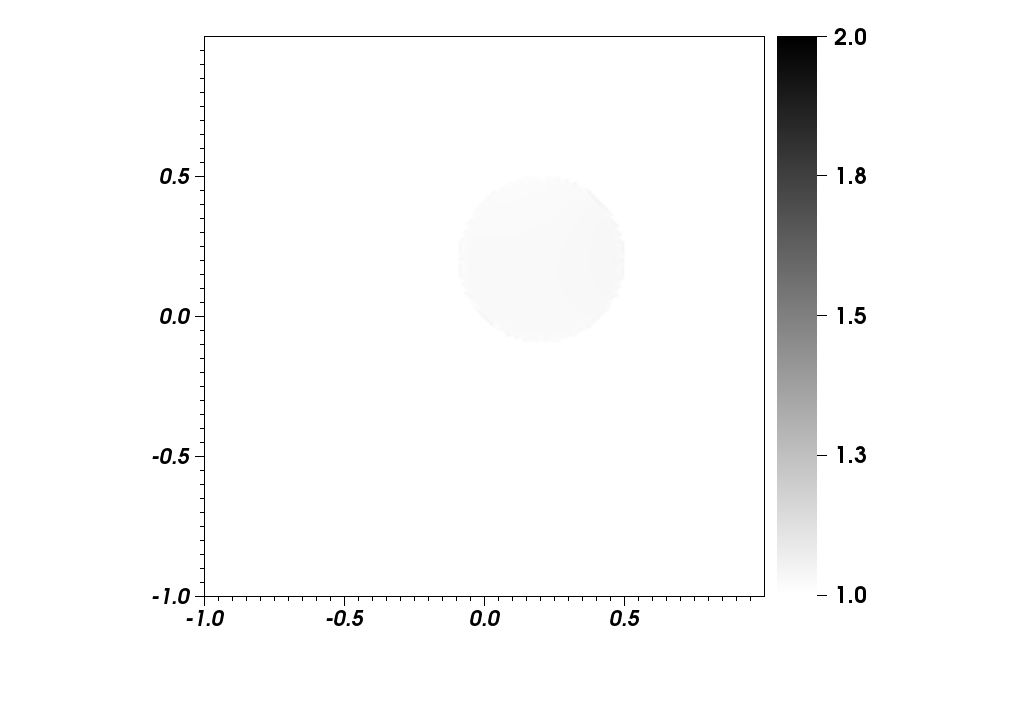}\label{disk_L2_7}}
    \subfigure[ $s=1,~\alpha=0.7$]{\includegraphics[scale=0.17,keepaspectratio]{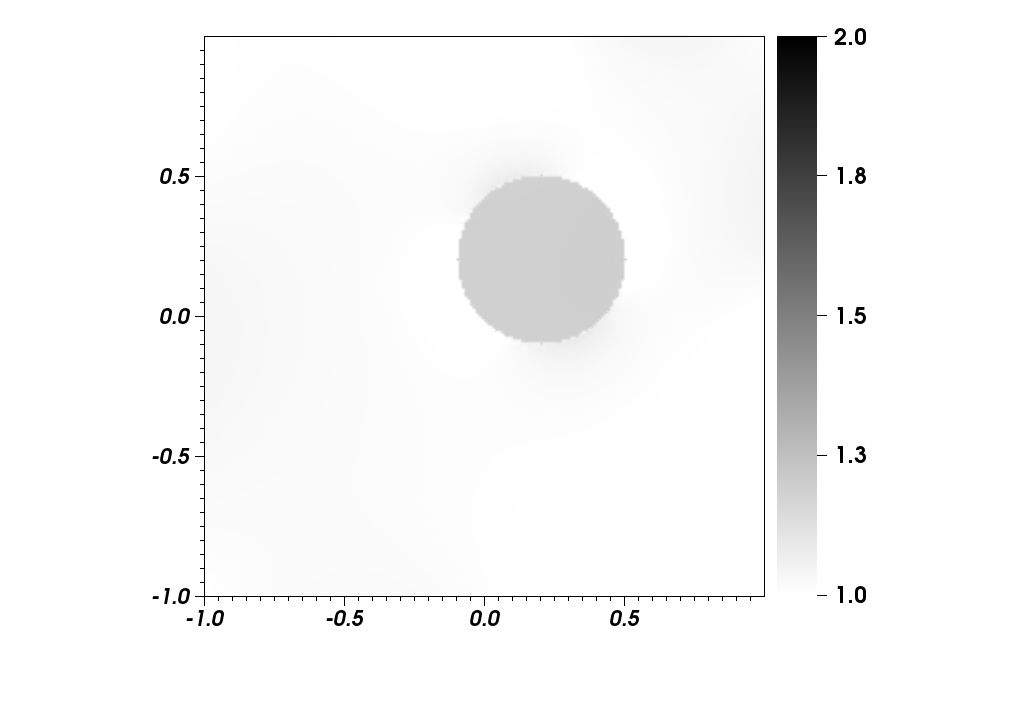}\label{disk_H1_7}}\\
	
	\caption{Test Case I: The actual and reconstructed Gaussian phantom for $\A=0.1$ with boundary condition BC1 and with the parametrix method. }
	\label{disk_sigma}
\end{figure}

Test Case II: In the second test case, we consider the heart and lung phantom for $\sigma$ as described in \cite{Mueller}. It has a background value of 1.0 that is perturbed in two ellipses (representing the lungs) where the value is 0.5 and in a circular region (representing the heart) where the value is 2.0. 

 In order to demonstrate the robustness of our optimization framework, we add $10\%$ and $25\%$ white Gaussian noise in the exact interior data $H$. The noise is added to $H$ in the following way: Let $\delta$ denote the noise level. Then 
 \begin{equation}\label{noisy data}
 H^\delta = H + \delta \cdot H \cdot N,
 \end{equation}
 where $H^\delta$ is the 2D-matrix of noisy data, $H$ is the 2D-matrix of data without noise and $N$ is the 2D-matrix of values each obtained from a standard normal distribution. In \eqref{noisy data}, the product refers to entrywise multiplication.
 
  The reconstructions of $\sigma$ with $L^{2}$ and $H^{1}$ regularizations 
  are shown in Figures \ref{heart_lung_L2} and \ref{heart_lung_H1}, respectively. The simulations show that our algorithm is very robust in the presence of noisy data. Furthermore, better contrast is obtained with the $H^{1}$ regularization term in comparison to the $L^{2}$ case.

\begin{figure}[h]
	\centering
	\subfigure[Actual phantom]{\includegraphics[scale=0.17,keepaspectratio]{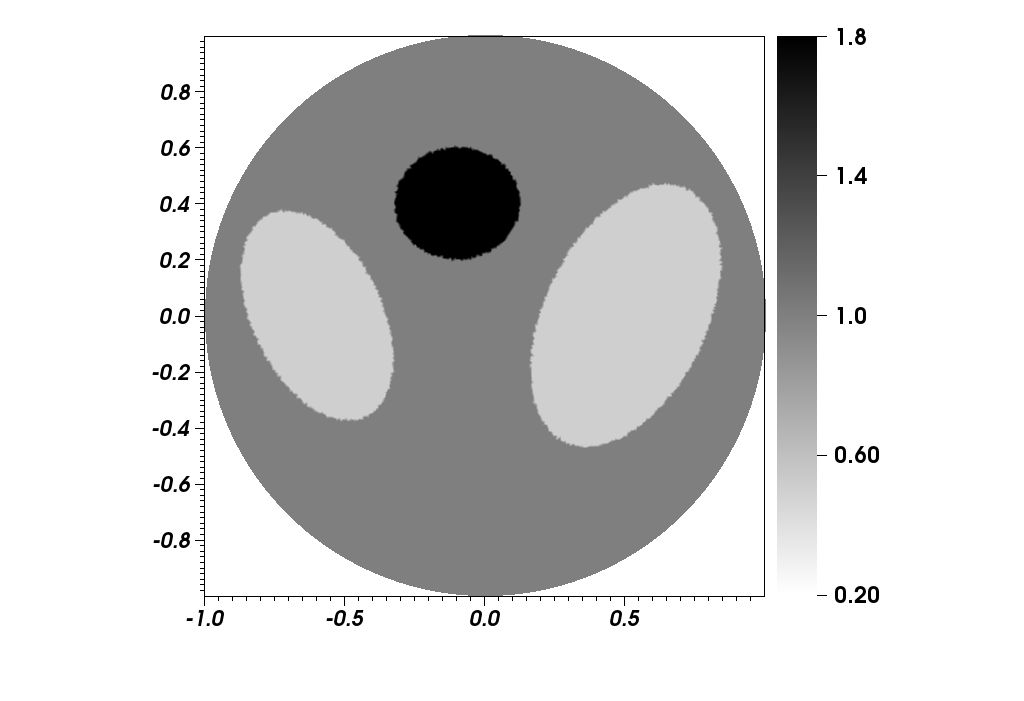}\label{heart_lung_actual}}\\
	\subfigure[ $s=0,~\alpha=0.1$]{\includegraphics[scale=0.17,keepaspectratio]{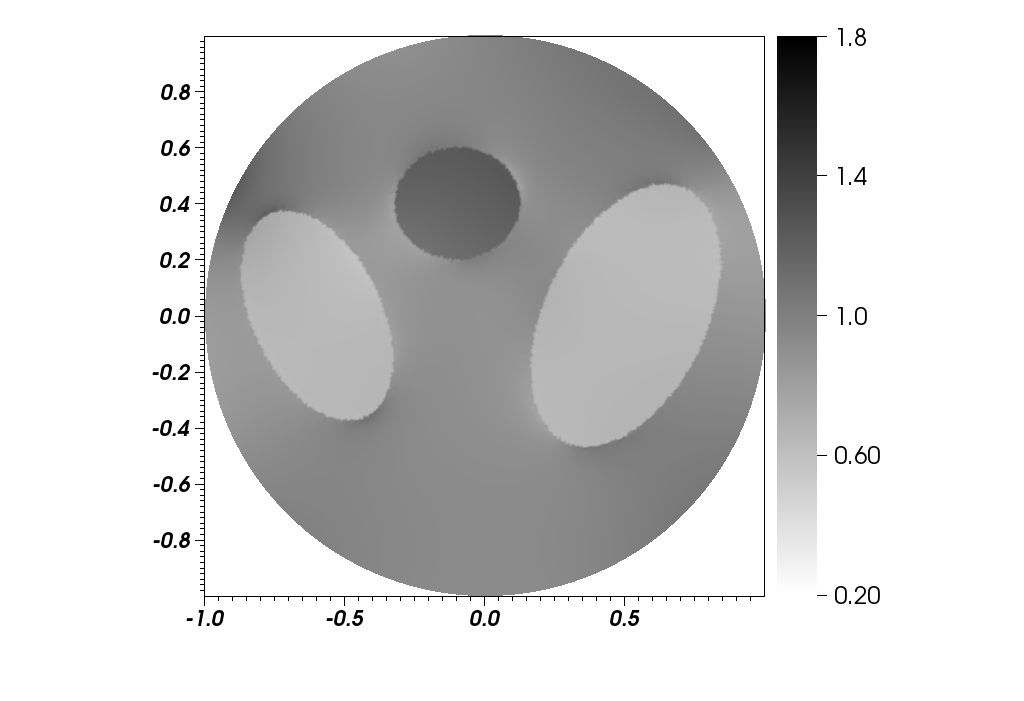}\label{heart_lung_L2_recon}} 
	\subfigure[ $s=1,~\alpha=0.1$]{\includegraphics[scale=0.17,keepaspectratio]{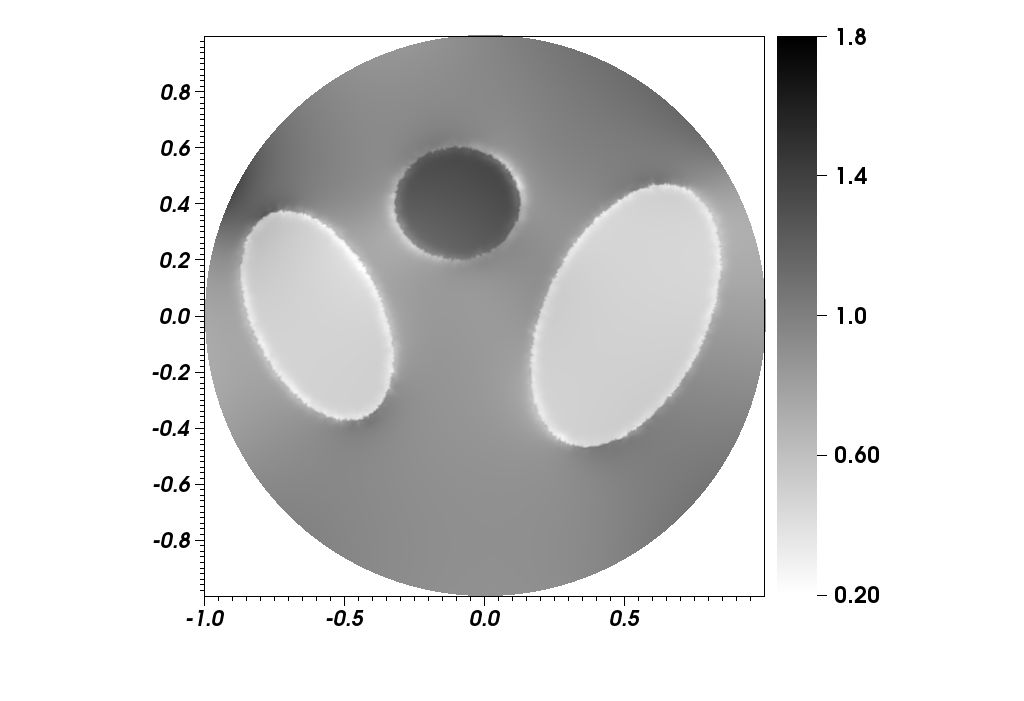}\label{heart_lung_H1}}\\
	\subfigure[$s=0,~\alpha=0.1$, $10\%$ noise]{\includegraphics[scale=0.17,keepaspectratio]{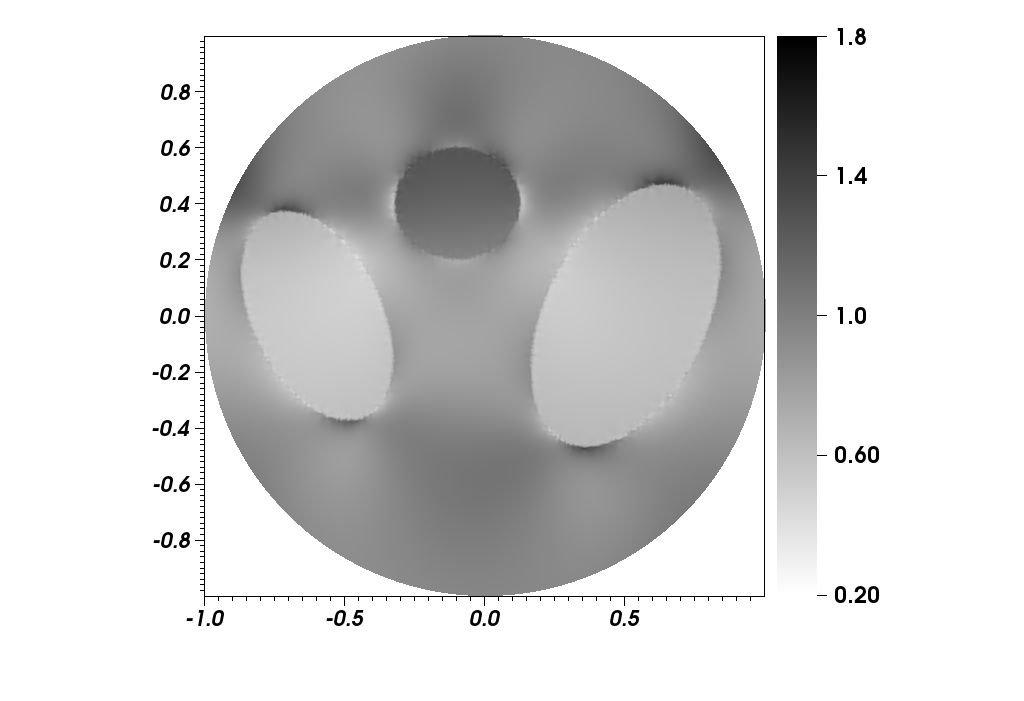}\label{heart_lung_noise_L2_10}}
	\subfigure[$s=1,~\alpha=0.1$, $10\%$ noise]{\includegraphics[scale=0.17,keepaspectratio]{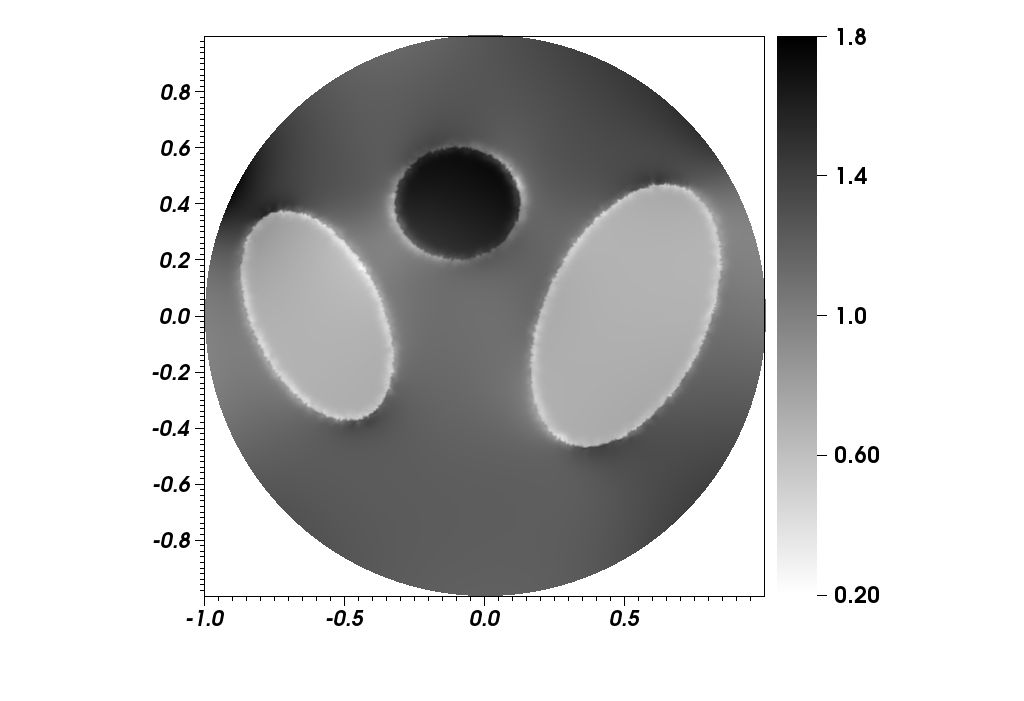}\label{heart_lung_noise_H1_10}}\\
	\subfigure[$s=0,~\alpha=0.1$, $25\%$ noise]{\includegraphics[scale=0.17,keepaspectratio]{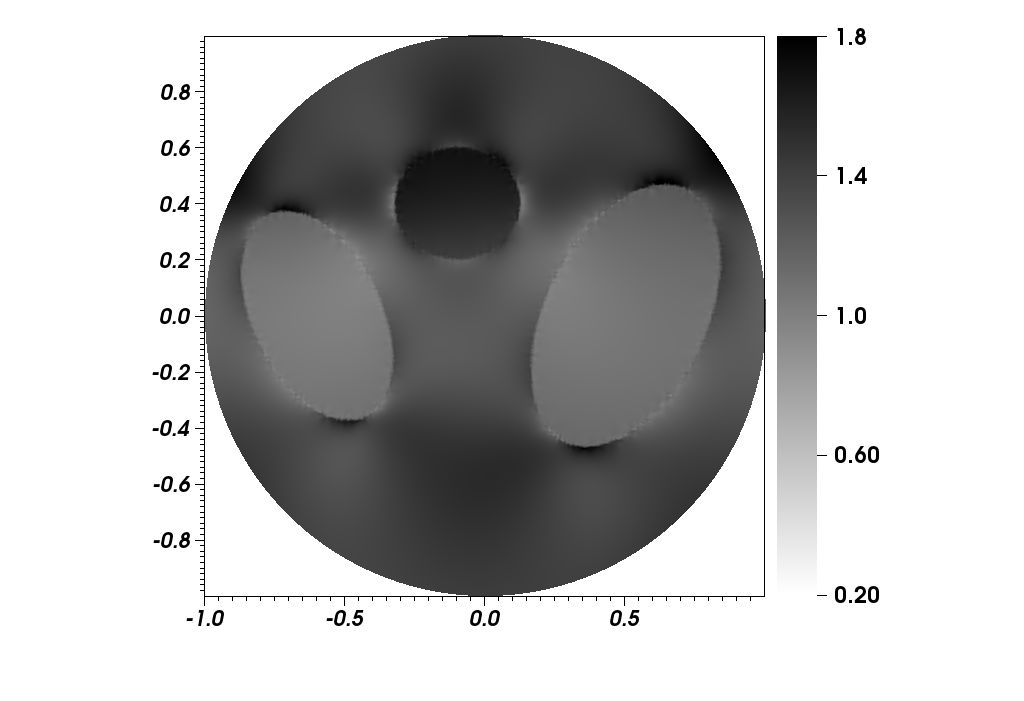}\label{heart_lung_noise_L2_25}}
		\subfigure[$s=1,~\alpha=0.1$, $25\%$ noise]{\includegraphics[scale=0.17,keepaspectratio]{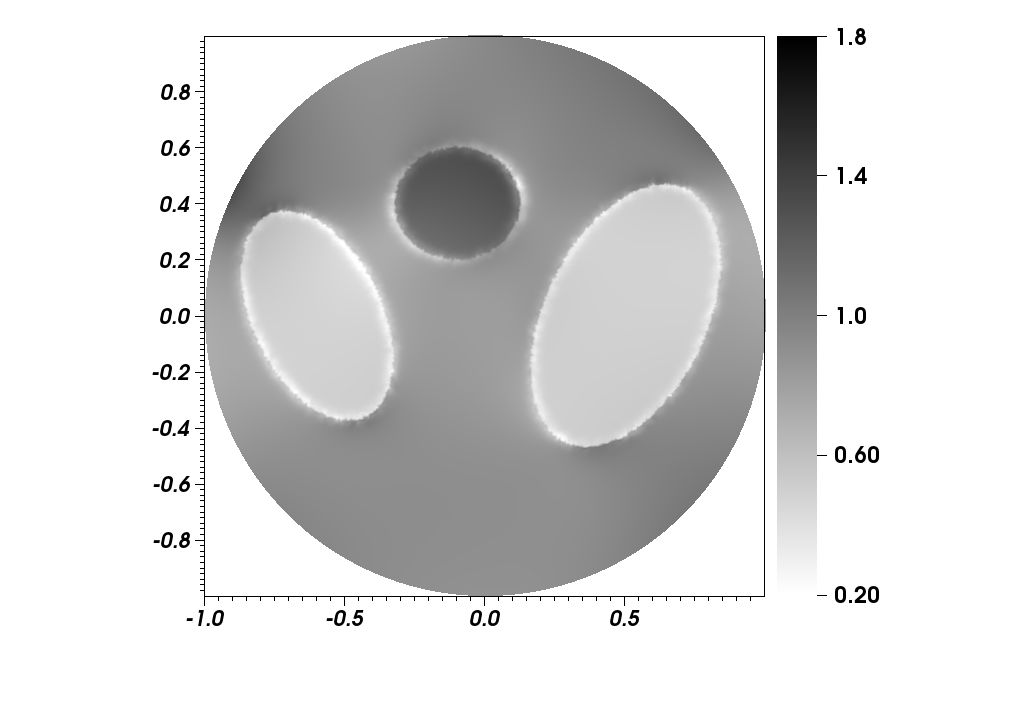}\label{heart_lung_noise_H1_25}}\\
	\caption{Test Case II: The actual and reconstructed heart and lung phantom with $L^{2}$ and $H^{1}$ regularizations and with noiseless/noisy data.}
	\label{heart_lung_L2}
\end{figure}

Test Case III: In the third test case, we consider a phantom where the conductivity $\sigma$ is supported inside a rotated rectangle
\begin{equation}
\sigma(x) =
\begin{cases}
&2.0,\qquad \mbox{ if }\Bigg|\frac{x}{\sqrt{2}}+\frac{y}{\sqrt{2}}-0.2\Bigg| < 0.2 \mbox{ and }\Bigg|\frac{x}{\sqrt{2}}-\frac{y}{\sqrt{2}}-0.2\Bigg| < 0.4,\\
&1.0,\qquad \mbox{ elsewhere.}\\
\end{cases}
\end{equation}
The reconstructions of $\sigma$ with $L^2$ and $H^1$ regularizations and with BC1, BC2 and BC3 given in (\ref{types_bc}) are shown in Figure \ref{rectangle_rotated}.

\begin{figure}[h]
\centering
\subfigure[Actual phantom]{\includegraphics[scale=0.17,keepaspectratio]{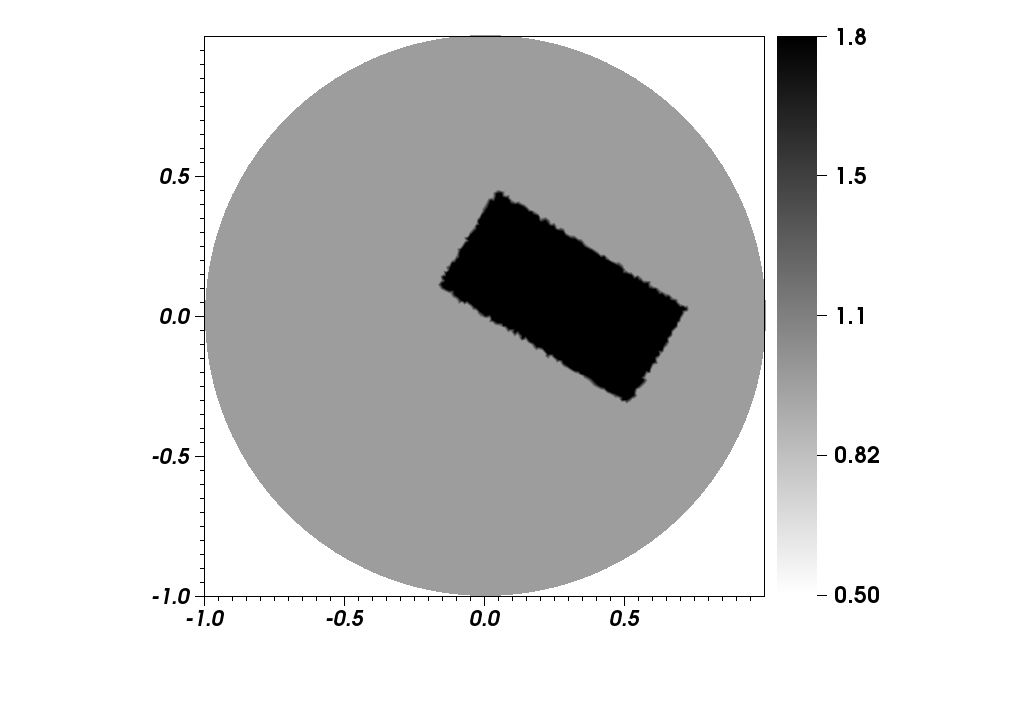}\label{rectangle_rotated_actual}}\\
\subfigure[$s=0,~\alpha=0.1$ with BC1]{\includegraphics[scale=0.17,keepaspectratio]{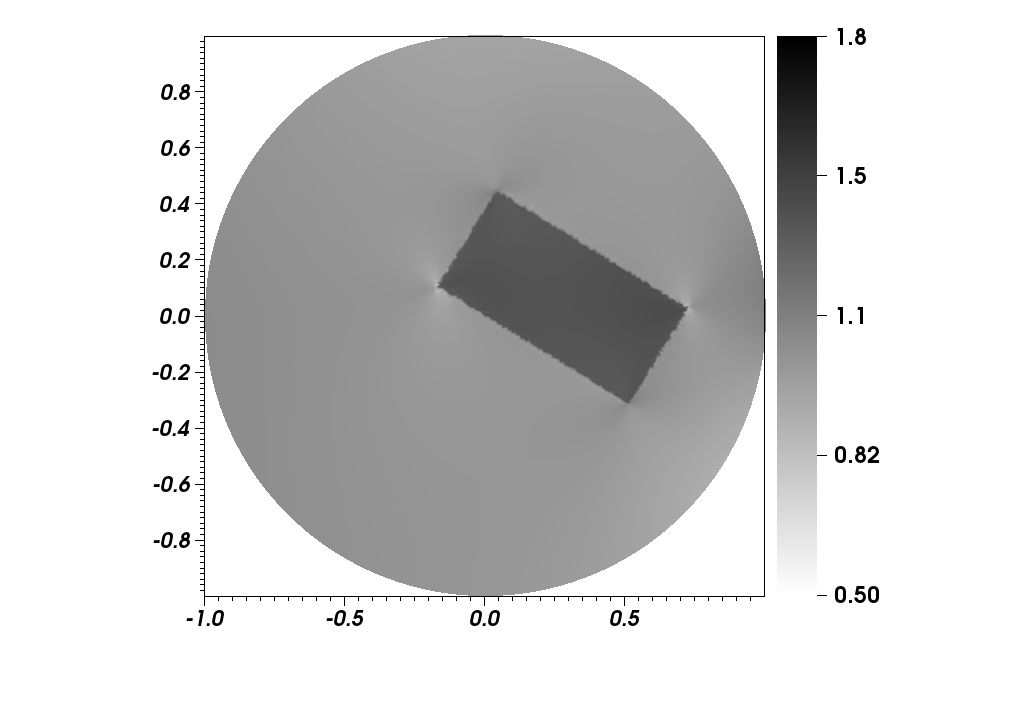}\label{rectangle_rotated_xx+yL2}}
\subfigure[$s=1,~\alpha=0.1$ with BC1]{\includegraphics[scale=0.17,keepaspectratio]{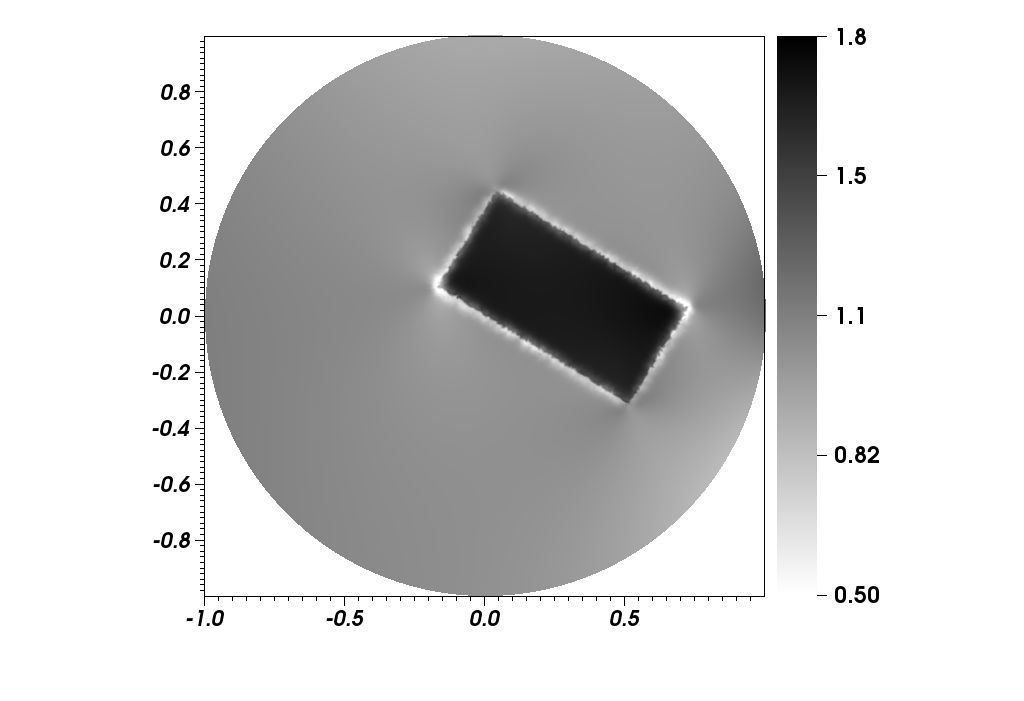}\label{rectangle_rotated_xx+y_H1}}\\
\subfigure[$s=0,~\alpha=0.1$ with BC2]{\includegraphics[scale=0.17,keepaspectratio]{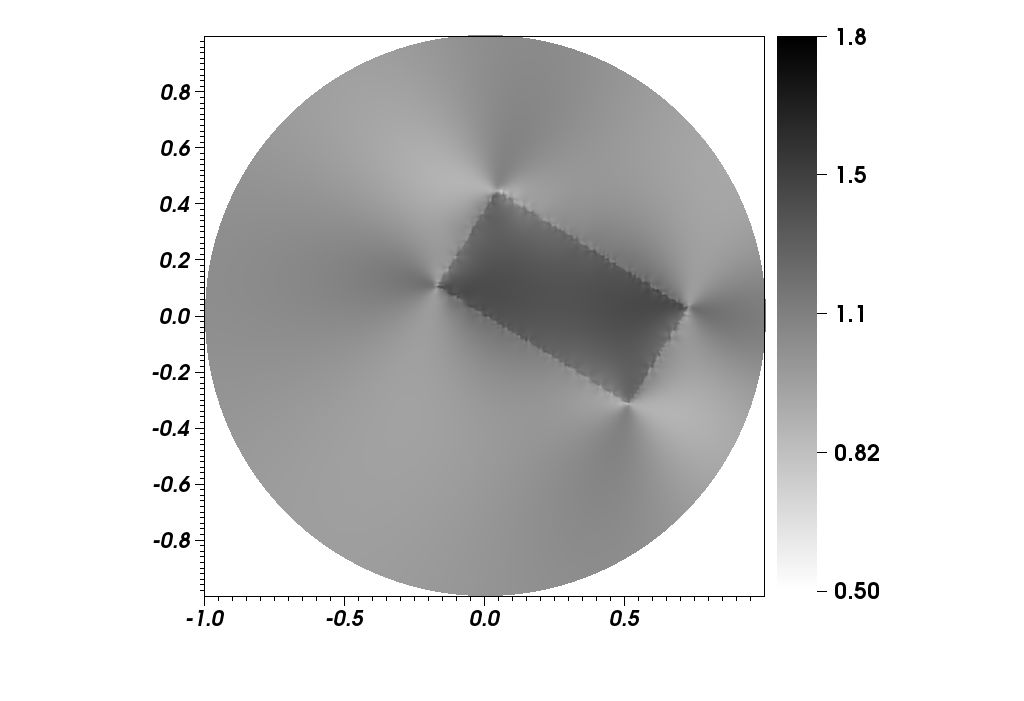}\label{rectangle_rotated_xy_L2}}
\subfigure[$s=1,~\alpha=0.1$ with BC2]{\includegraphics[scale=0.17,keepaspectratio]{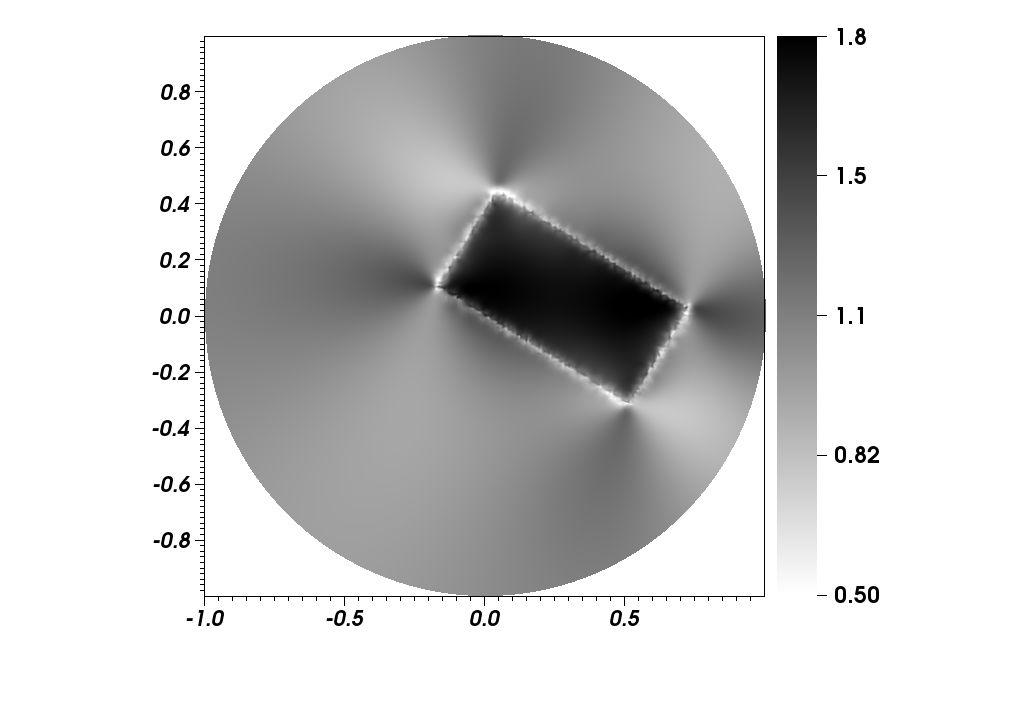}\label{rectangle_rotated_xy_H1}}\\
\subfigure[$s=0,~\alpha=0.1$ with BC3]{\includegraphics[scale=0.17,keepaspectratio]{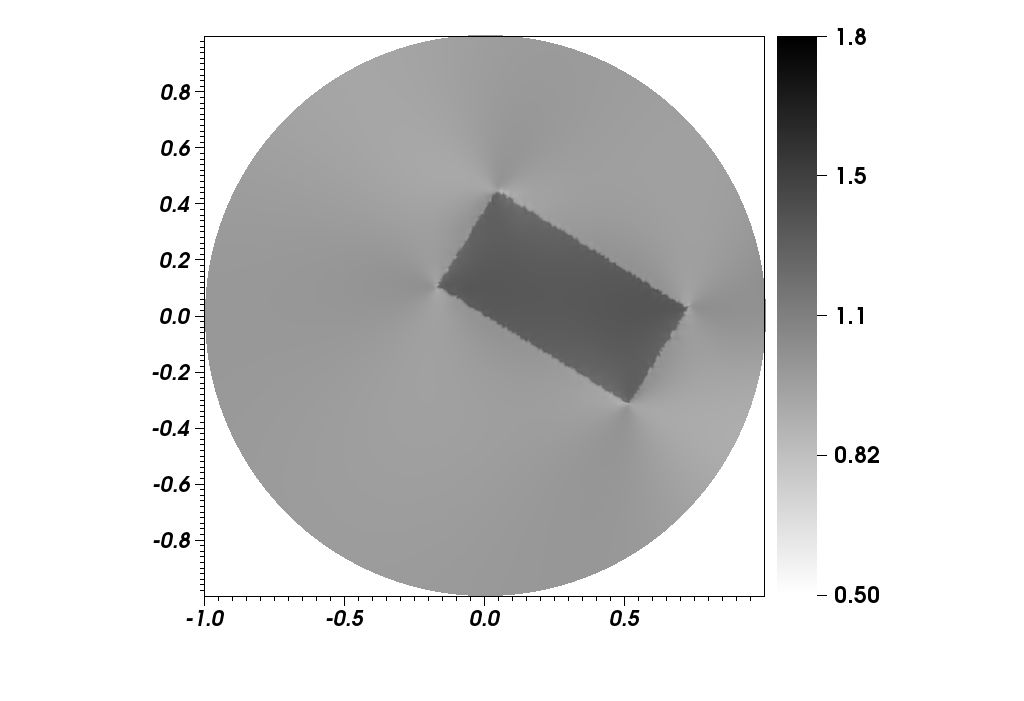}\label{rectangle_rotated_3bc_L2}}
\subfigure[$s=1,~\alpha=0.1$ with BC3]{\includegraphics[scale=0.17,keepaspectratio]{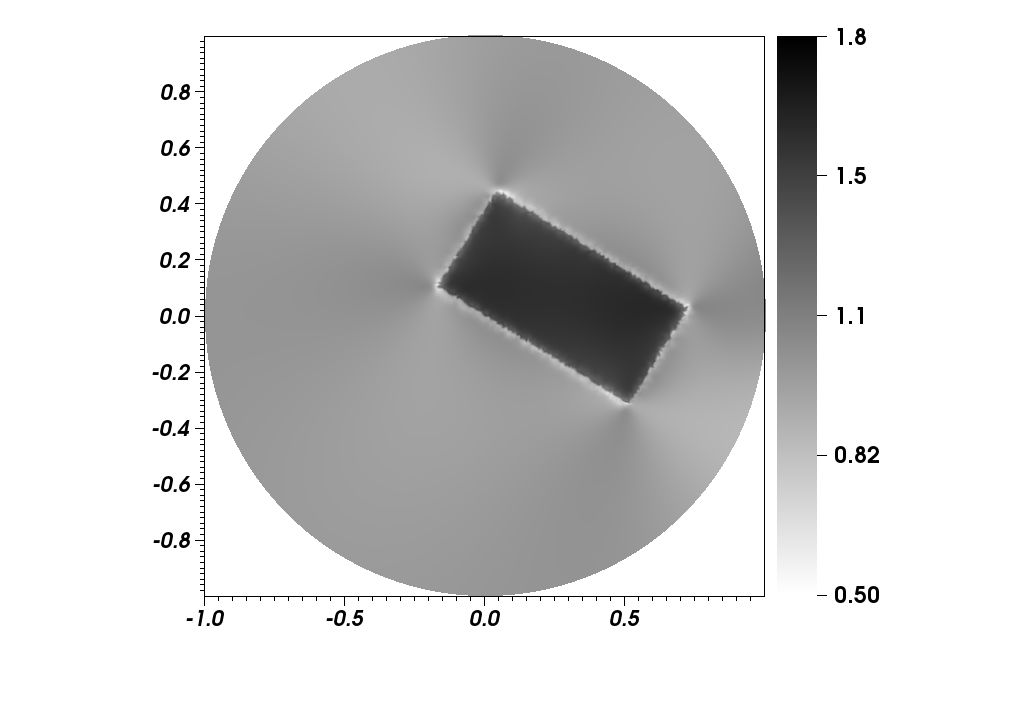}\label{rectangle_rotated_3bc_H1}}\\
    \caption{Test Case III: The actual and reconstructed rotated rectangle phantom with $L^{2}$ and $H^{1}$ regularization, and  with the boundary conditions BC1, BC2 and BC3. }
    \label{rectangle_rotated}
  \end{figure}

 {Our goal with this simulation is to show the effect of boundary conditions on the reconstructed images. As we can see, the reconstructions with boundary conditions BC1 and BC3 are better compared to the ones with BC2,  in the sense that there are fewer artifacts with BC1 and BC3. We note that a similar behavior was previously observed and studied theoretically and numerically for the linearized reconstruction method \cite{Hoffmann_Knudsen,BHK} using microlocal analysis. The characterization of artefacts appearing  in reconstructions from the fully non-linear algorithm is, in our opinion, non-trivial, and beyond the scope of the current work.  
Note further that BC1 with only two boundary conditions yields reconstructions similar in quality to the reconstructions from BC3 with three boundary conditions. This makes us conjecture that the non-linear reconstruction problem in AET is solvable with only two properly chosen boundary conditions.}

{Test Case IV: In the fourth test case, we consider a combination of phantoms supported in a square $S_a = \lbrace (x,y) \in \mathbb{R}^2: -0.1 < x < -0.1, -0.1 < y < -0.1  \rbrace $ with $\sigma=3.0$, 2 disks centered at $(-0.1,0.5)$ with radius 0.2 and $\sigma=2.0$ and at $(0.1,0.5)$ with radius 0.2 and $\sigma=1.0$ and a bean-shaped annulus with value of $\sigma=2.0$ in the annular region and $\sigma=0.5$ in the hole. The plots of the reconstructed $\sigma$ for $s=0,~\alpha=0.1$ and $s=1,~\alpha=0.1$ with BC1 as given in (\ref{types_bc}) and with the parametrix method are shown in Figure \ref{clown_phantom}.}
\begin{figure}[h]
\centering
\subfigure[Actual phantom]{\includegraphics[scale=0.17,keepaspectratio]{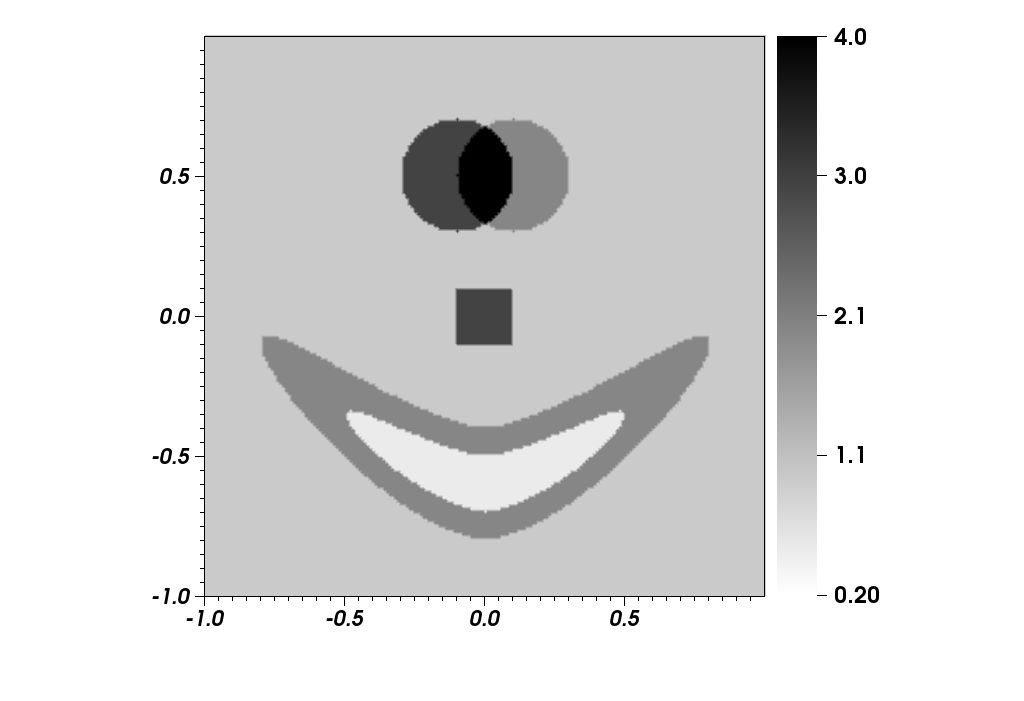}\label{clown_exact}}\hspace{5mm}
\subfigure[Paramterix method \cite{Kuchment-Kunyansky-AET} (Done in Matlab)]{\includegraphics[height=1.65in,width=1.95in]{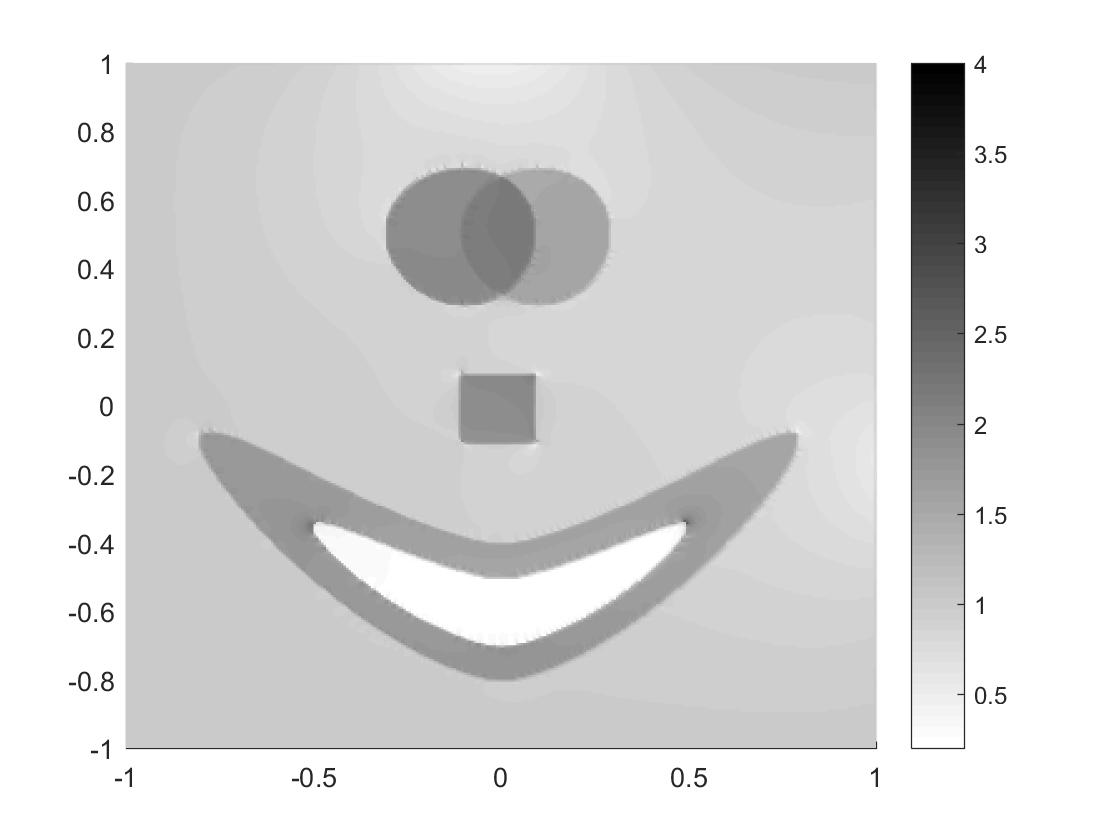}\label{clown_paramterix}}\\
\subfigure[$s=0,~\alpha=0.1$]{\includegraphics[scale=0.17,keepaspectratio]{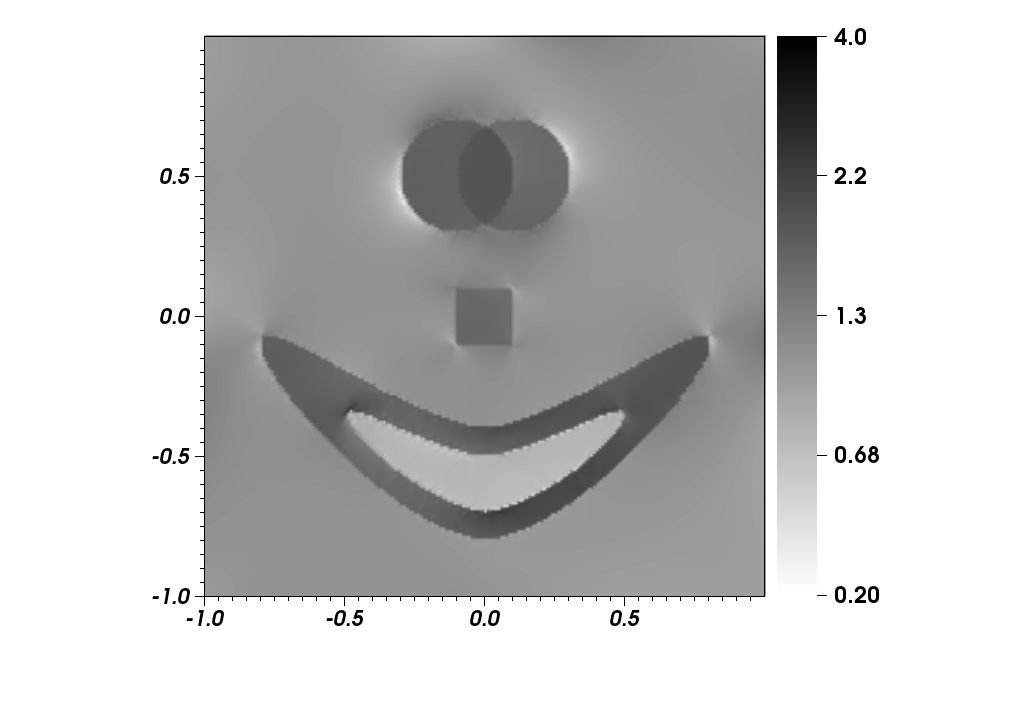}\label{clown_L2}}
\subfigure[ $s=1,~\alpha=0.1$]{\includegraphics[scale=0.17,keepaspectratio]{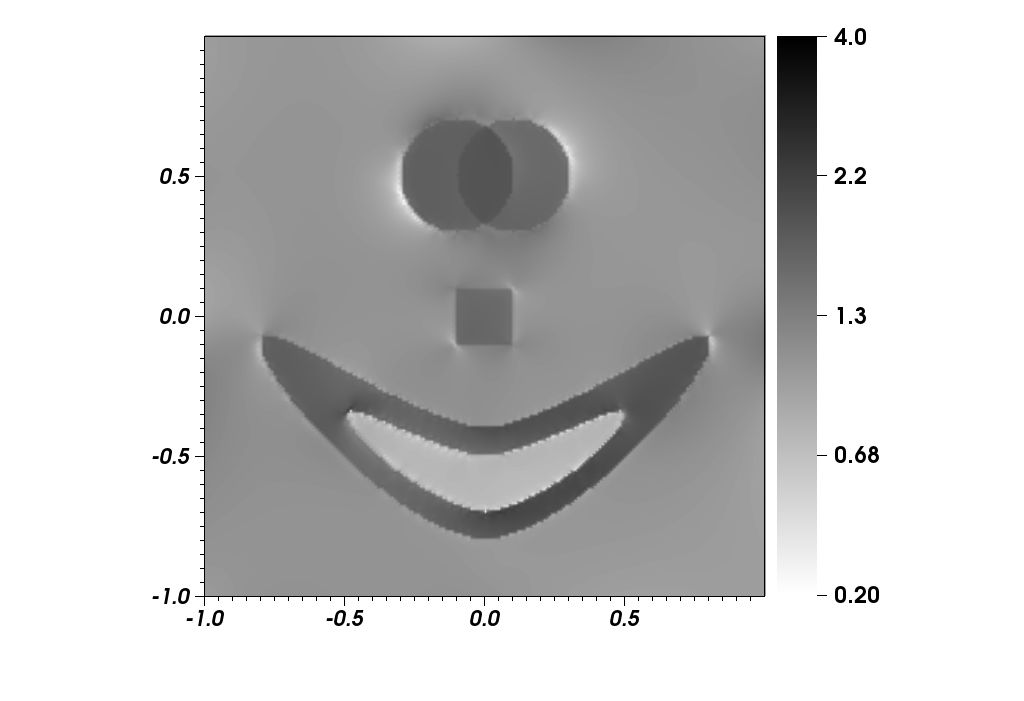}\label{clown_H1}}\\
    \caption{Test Case IV: The actual and reconstructed images of a phantom with an inclusion as well as with self-intersections.}
    \label{clown_phantom}
  \end{figure}

{Our numerical procedure performs well for a phantom with an inclusion as well as with a  self-intersection as shown in Figure \ref{clown_phantom}. Note that the inclusion is clearly visible. Furthermore for the two disks with intersections, the intersecting region is clearly distinguishable as well. We compare with the parametrix method of \cite{Kuchment-Kunyansky-AET} (see Figure \ref{clown_paramterix}), and similar to what was observed in the case of Figure \ref{disk_sigma}, there is a substantial loss of contrast with the parametrix method.}

\section{Conclusion}\label{sect:conclusion}
In this work, we considered a non-linear computational approach to acousto-electric tomography involving the reconstruction of the electric conductivity of a medium from interior power density distribution. We formulated the inverse problem as a non-linear optimization problem, showed the existence of a minimizer and developed a non-linear conjugate gradient (NLCG) scheme for the reconstruction of the conductivity of the medium from interior power density functionals. We presented several numerical simulations showing the robustness of the NLCG algorithm. We observed that the $H^1$ regularization, in general, reconstructed images with better contrast compared to the $L^2$ regularization which reconstructed images with better resolution. The proposed non-linear framework is versatile and can be applied to other hybrid imaging modalities as well.

\section*{Acknowledgements}
Knudsen would like to acknowledge support from the Danish Council for Independent Research | Natural Sciences. Krishnan was supported in part by US NSF grant DMS 1616564.
Additionally, he and Roy benefited from the support of Airbus Corporate Foundation Chair grant titled
``Mathematics of Complex Systems'' established at TIFR CAM and TIFR ICTS,
Bangalore, India.

\bibliographystyle{plain}
\bibliography{References_bib_file}
\end{document}